\documentclass[paper=a4, fontsize=12pt]{article}
\usepackage[utf8]{inputenc}
\usepackage[T2A]{fontenc}

\usepackage [english]{babel}
\usepackage[a4paper,hmargin=2.5cm,vmargin=2.5cm]{geometry}

\usepackage {pgfpict2e}
\usepackage{amssymb}
\usepackage{amsthm}
\usepackage{amsmath}
\usepackage{amscd}
\usepackage{bbm}
\usepackage{setspace}
\usepackage{mathrsfs}
\usepackage{graphicx}
\usepackage{dsfont}
\usepackage{hyperref}
\usepackage{graphicx}
\usepackage{tikz-cd}
\usepackage{comment}

\hypersetup{pdfstartview=FitH, linkcolor=blue, urlcolor=blue, citecolor=blue, colorlinks=true}

\DeclareMathOperator{\CB}{CB}
\DeclareMathOperator{\Split}{Split}
\DeclareMathOperator{\Ima}{Im}
\DeclareMathOperator{\Ker}{Ker}
\DeclareMathOperator{\Hom}{Hom}
\DeclareMathOperator{\Tr}{Tr}

\newcommand{\sltwo}{{\mathfrak{sl}_2}}

\newcommand{\Shares}{\mathcal{S}}

\theoremstyle{plain}
\newtheorem{Claim}{Claim}
\newtheorem{Def}[Claim]{Definition}
\newtheorem{corollary}[Claim]{Corollary}
\newtheorem{Lemma}[Claim]{Lemma}

\newtheorem{Theorem}[Claim]{Theorem}
\newtheorem*{question*}{Question}

\newtheorem*{Remark*}{Remark}

\numberwithin{Claim}{section}
\numberwithin{equation}{section}

\title{Duality for the $\mathfrak{sl}_2$ weight system}
\author{Polina Zakorko\thanks{Higher School of Economics, Israel Institute of Technology, pezakorko@edu.hse.ru}, Polina Zinova\thanks{Higher School of Economics, apoly38@gmail.com}}
\date{}

\begin{document}

\maketitle
\tableofcontents

\section{Introduction}

A \textbf{chord diagram of order $n$} (a chord diagram with $n$ chords) is an oriented circle together with $2n$ pairwise distinct points split into $n$ disjoint pairs considered up to orientation preserving diffeomorphisms of the circle.
In pictures below, we connect points of the same pair by a segment of a line or of a curve, called a \textbf{chord}. 

Weight systems are functions on chord diagrams that satisfy the so-called four-term relations. 
In terms of weight systems one can express knot invariants of finite type introduced by Vassiliev \cite{V90} around 1990. 
Moreover, over a field of characteristic zero every weight system corresponds to some invariant of finite order. 
This was proved by Kontsevich \cite{Ko}.

Around 1995, Bar-Natan \cite{BN} and Kontsevich \cite{Ko} associated a weight system to any finite-dimensional Lie algebra endowed with a nondegenerate invariant bilinear form. 
Such a weight system takes values in the center of the universal enveloping algebra of the Lie algebra. 
The simplest non-trivial example of such a weight system is the $\sltwo$ weight system. 
The knot invariant to which this weight system corresponds is the colored Jones polynomial.

The value of the $\sltwo$ weight system on a chord diagram with $n$ chords is a monic polynomial of degree~$n$ in $c$, where $c\in U(\sltwo)$ denotes the Casimir element of the Lie algebra $\sltwo$. 
It is a hard task to compute values of this weight system because one has to make
computations in a noncommutative algebra.  
Chmutov--Varchenko recurrence relations~\cite{ChV} significantly simplify the computations; however, using them is still laborious due to exponentially growing number of 
chord diagrams involved in the recursion process. 
Hence, computation of the value of this weight system on a particular chord diagram turns out to be a difficult problem in which there have been no significant progress for a long time and such advances have only been achieved recently.

To each chord diagram, its \textbf{intersection graph} is assigned. The vertices of this graph correspond one-to-one to the chords of the chord diagram, and two vertices are connected by an edge if and only if the corresponding chords intersect one another
(that is, their ends alternate). At the same time, there exist simple graphs that are not intersection graphs of any chord diagram. According to the Chmutov--Lando theorem \cite{ChL}, the value of the $\sltwo$ weight system on a chord diagram depends only on its intersection graph. Therefore, we may talk about the values of this weight system on intersection graphs. Lando posed a problem whether
there exists an extension of the $\sltwo$ weight system
to a polynomial graph invariant satisfying $4$-term
relations for graphs. Up to now, this question is answered
in affirmative only for graphs with up to~$8$ vertices~\cite{K21}. 

One of the ways to approach 
this problem in full generality
consists in computing explicit values
of the $\sltwo$ weight system on large families of intersection
graphs, so that a presumable extension could be conjectured.
A serious progress has been achieved in this direction during
the last years.
Recently, the first author proved an explicit formula for the values of the $\sltwo$ weight system on complete graphs \cite{Zakorko}. 
The formula was earlier conjectured by S.~Lando.
In addition,  M.~Kazarian and the second author~\cite{KaZi} 
deduced and proved an explicit formula for the values of this weight system on complete bipartite graphs.

The \textbf{join} of two graphs $\Gamma_1$, $\Gamma_2$ is the graph, denoted by $(\Gamma_1,\Gamma_2)$, in which
     $V((\Gamma_1,\Gamma_2)) = V(\Gamma_1)\sqcup V(\Gamma_2)$ and     
 $E((\Gamma_1,\Gamma_2)) = E(\Gamma_1) \sqcup E(\Gamma_2) \sqcup (V(\Gamma_1) \times V(\Gamma_2))$. Here and below we denote by $V(\Gamma)$ the vertex set of a graph $\Gamma$, and by $E(\Gamma)$ its edge set. In particular, any complete bipartite graph is the join of two empty (discrete) graphs.
 We denote the join of a graph~$\Gamma$ with a discrete graph
 on~$n$ vertices by $(\Gamma,n)$.

Now, let $\Gamma$ be a graph such that $(\Gamma,n)$ is
an intersection graph for each $n=0,1,2,\dots$. It
can be easily seen that this is the case if the first two
of the graphs in this sequence, namely, $\Gamma=(\Gamma,0)$
and $(\Gamma,1)$ are intersection graphs.
It was proved in~\cite{KaZi} that 
for such a graph~$\Gamma$ there is a sequence of polynomials
$r_0^{(\Gamma)}(c),r_1^{(\Gamma)}(c),\dots,r_{|V(\Gamma)|}^{(\Gamma)}(c)$ such that
the value of the $\sltwo$ weight system on the intersection graphs $(\Gamma,n)$ has the form
\begin{equation}
    \label{eq:fromKaZi}
   w_\sltwo((\Gamma,n))= \sum_{k=0}^{|V(\Gamma)|} r_k^{(\Gamma)} (c) \left(c-\frac{k(k+1)}{2}\right)^n,
\end{equation} 
so that 
$$
G_\Gamma(t)=\sum_{k=0}^{|V(\Gamma)|} \frac{r_k^{(\Gamma)} (c)}{1- \left(c-\frac{k(k+1)}{2}\right)t},
$$
where $G_\Gamma(t)=\sum_{n=0}^\infty w_\sltwo((\Gamma,n))t^n$ is the generating function
for the values of the $\sltwo$ weight system on the joins $(\Gamma,n)$
of~$\Gamma$ with discrete graphs.
Here $|V(\Gamma)|$ is the number of vertices of $\Gamma$, and $r_{k}^{(\Gamma)}(c)$ is a polynomial in $c$ of degree at most $k$,
$k=0,1,\dots,|V(\Gamma)|$. 
Some examples of such values for all the graphs $\Gamma$ with at most $4$ vertices are given in \cite{Fil22}.

The main result of the present paper consists in the proof of the following statement conjectured by S.~Lando, 
which establishes a duality between the values of the $\sltwo$
weight system on two sequences of joins defined by a graph~$\Gamma$
and its complement graph~$\overline\Gamma$:
\begin{Theorem} 
    \label{thm:main_graphs}
    Let $\Gamma$ be a simple graph such that $(\Gamma,1)$ is an intersection graph.
    Let $\overline \Gamma$ be the complement graph of $\Gamma$, i.e., the graph such that
    its set of vertices coincides with that of~$\Gamma$
    and its set of edges is complementary to the set of edges of $\Gamma$. Then in the notation of Eq.~\eqref{eq:fromKaZi}
    we have 
    $$r_k^{(\overline \Gamma)} = (-1)^{|V(\Gamma)|-k}\cdot r_k^{(\Gamma)}.$$
\end{Theorem}

Hence, knowing the polynomials~$r_k$ for a graph, we immediately
reconstruct them for the complement graph. For example,
in~\cite{KaZi} these polynomials have been computed for all
the discrete graphs, so that their joins with discrete graphs
are complete bipartite graphs~$K_{m,n}$.

Then Theorem~\ref{thm:main_graphs} implies

\begin{corollary}
    For the complete split graphs $(K_m,n)$ we have
    $$r_k^{(K_m)}=(-1)^{m-k}r_k^{(K_1^m)},$$ where $K_1^m$
    is the discrete graph on~$m$ vertices.
\end{corollary}

For example, for $m=3$ we have
$$
\sum_{n=0}^\infty w_\sltwo(K_{3,n})t^{3+n}=
\frac{c t^3}{30} \left(\frac{5 c}{1 - c  t} + \frac{6(3 c^2 - 2 c + 
      2)}{1 - (c - 1) t} + \frac{10(4 c - 3)}{1 - (c - 3) t}
      + \frac{3(4 c^2 - 11 c + 6)}{1 - (c - 6) t}\right),
$$
which  yields
$$
\sum_{n=0}^\infty w_\sltwo((K_3,n))t^{3+n}=\frac{c t^3}{30} \left(-\frac{5 c}{1 - c  t} + \frac{6(3 c^2 - 2 c + 
      2)}{1 - (c - 1) t} - \frac{10(4 c - 3)}{1 - (c - 3) t}
      + \frac{3(4 c^2 - 11 c + 6)}{1 - (c - 6) t}\right).
$$

Another important corollary of Theorem~\ref{thm:main_graphs} is
that it produces restrictions for graphs isomorphic to
their complement graphs:

\begin{corollary}
    If $\Gamma\cong \overline\Gamma$, then $r^{(\Gamma)}_k=0$
   if $|V(\Gamma)|-k$ is odd.
\end{corollary}

An example of self-complementary intersection graph is given by the graph 
$\Gamma:={\includegraphics[width=20pt]{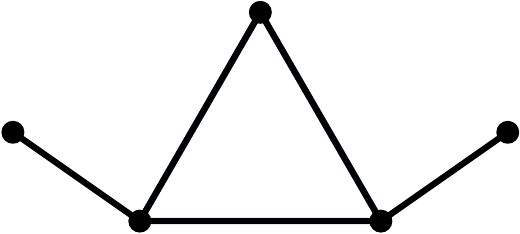}}$ on~$5$ vertices. The joins of this graph with discrete graphs are intersection graphs. The polynomials~$r_i^{(\Gamma)}(c)$
for it are
\begin{align}
    r_1^{(\Gamma)}(c)&=\frac{1}{70}\left(30 c^5-60 c^4-13 c^3+22 c^2+8 c\right),\\ r_3^{(\Gamma)}(c)&=\frac{1}{45}\left(20 c^5-115 c^4+123 c^3+108 c^2-108 c\right),
    \\
    r_5^{(\Gamma)}(c)&=\frac{1}{126}\left(16 c^5-200 c^4+813 c^3-1224 c^2+540 c\right),
\end{align}
and $r_i^{(\Gamma)}(c) = 0$ for $i$ even or greater than $5$. 
It is interesting to remark that in spite of the fact that
the other self-complimentary graph on five vertices,
the cycle $C_5$, does not produce a sequence
of intersection graphs (its join $(C_5,1)$ with the one-vertex graph is an intersection graph no longer), the
generating function for the values of the $\sltwo$
weight system corresponding to this graph still possesses
the property in the Corollary. These values were computed
in~\cite{Fil22} under the assumption that the desired extension
of the $\sltwo$ weight system to arbitrary graphs exists.
The even indexed $r$-polynomials in this case also proved to be~$0$.

The main tool in studying the values of the $\sltwo$ weight system on series of graph joins is the generalization of the notion of chord diagram which is called \textbf{chord diagram on $k$ strands} (see \cite{Chmutov_Duzhin_Mostovoy}).
Similar to the original chord diagram, a chord diagram on k strands has chords, but their endpoints are located on $k$ parallel oriented numbered lines, not on a circle. 
We consider mainly chord diagrams on $2$ strands. 
They can be referred to as \textbf{shares}.

The notion of a weight system as well as the construction of it from a Lie algebra $\mathfrak{g}$ can be generalized to chord diagrams on any number of strands.
For the $\sltwo$-weight system on the vector space over $\mathbb{C}$ spanned by the chord diagrams on $k$ strands we still have the Chmutov-Varchenko relations. 
We prove that the $\sltwo$ weight system on chord diagrams on $2$ strands $w_{\sltwo}$ takes values in the algebra of polynomials $\mathbb{C}[x,c_1,c_2]$, where $c_i$ stands for the image of a diagram with one chord on the $i$th strand.
We work with the algebra $\Shares$ spanned by shares modulo the kernel of $w_{\sltwo}$.
This algebra, together with three bases for it which we denote 
respectively by $\{x^n\}, \{y^n\}, \{p_n\}$, $n = 0,1,2,\ldots$, 
as well as the operators $U, X, Y$ on $\Shares$ that add one chord to a share in different ways were studied in \cite{KaZi} and \cite{Zakorko}. 
We enhance these study in the present paper.

We introduce the fourth basis in $\Shares$, namely $\{e_n(c,y)\}$,
which is an eigenbasis of the operator $U$. 
It plays an important role in the proof of the main theorem.

The paper is organized as follows. 
In Section~\ref{sec:chord_diagrams} we recall the definitions and main properties of chord diagrams and of the $\sltwo$ weight system on chord diagrams.
In Section~\ref{sec:chord_diagrams_k}, we provide the definition of a chord diagram on $k$ strands, of the $\sltwo$ weight system on shares and of the algebra $\Shares$ of shares.
In Section~\ref{sec:U_and_eigen} we introduce operators $U, X, Y$ acting on the algebra of shares. We also discuss the main properties of the eigenbasis $e_n(c,y), n = 0,1,2,\ldots$ of the operator $U$.
In Section~\ref{sec:bilinear} we introduce a bilinear form on $\Shares$. 
The basis $\{p_n\}$ turns out to be orthogonal with respect to this bilinear form. We use this bilinear form to obtain an explicit formula for $e_n(c,c)$. This is important because the decomposition $I = \sum_{n=0}^{k} \alpha_k^I e_n(c,y)$
leads to the explicit formula
$\sum_{m=0}^k \frac{ \alpha_m^I(c) e_m(c,c)}{1-(c-n(n+1)/2) t}$
for the generating function $G_I(t)$. 

In Section~\ref{sec:involution} we introduce an involution $\sigma$ on $\Shares$, which we apply later
in Section~\ref{sec:proof_main} to prove main Theorem~\ref{thm:main_graphs}.
In Section~\ref{sec:props_e_n} we discuss some additional properties of the basis $\{e_n\}$.
In Section~\ref{sec:rel_to_complete} we prove a corollary of this theorem related to the values of the $\sltwo$ weight system on complete graphs.

\textbf{Acknowledgements}:
The authors are grateful to professor M.~Kazarian, professor S.~Chmutov and to professor S.~Lando for their permanent attention to this work and for useful advice. 
The second author was partially supported by the Foundation for the Advancement of Theoretical
Physics and Mathematics “BASIS” and by the Russian Science Foundation (grant 24-11-00366).

\textbf{Keywords}: chord diagram, chord diagram on $2$ strands, $\mathfrak{sl}_2$ weight system, intersection graph of a chord diagram, complement graph
\section{Notation} \label{sec:notation}
\begin{tabular}{c p{0.8\textwidth}}
    $\mathscr G$ & the space spanned by all the simple graphs modulo the four-term relation; \\
    $C$ & the space spanned by all the chord diagrams modulo the four-term relation; \\
    $A_k$ & the algebra generated by the chord diagrams on $k$ strands; \\
    $\Shares$ & quotient algebra of $S$ modulo the two-term, four-term, and six-term relations. 
    This algebra is isomorphic to the space of polynomials in $c,x$; \\ 
    $w_{\sltwo}$ & $\sltwo$ weight system on chord diagrams or intersection graphs; \\
    $\mathds{1}$ & the empty share; \\
    $(I,H)$ & a join of two shares $I$ and $H$, that is, a chord diagram obtained by 
    intersecting these shares and gluing their strands into one circle. 
    We use a similar notation for the join of two graphs $\Gamma_1, \Gamma_2$, i.e., $(\Gamma_1, \Gamma_2)$ is the graph obtained by adding all edges connecting vertices of $\Gamma_1$ with vertices of $\Gamma_2$ to the disjoint union $\Gamma_1 \sqcup \Gamma_2$); \\
    $\sigma$ & the involution on $\Shares$; \\
    $\langle \cdot, \cdot \rangle$ & the non-degenerate bilinear form on $\Shares$ that maps a pair of shares $I,H$ to the value of $w_{\sltwo}$ on the join $(I,H)$, i.e., $\langle \cdot, \cdot \rangle\colon \Shares \times \Shares \to \mathbb{C}[c], \langle I, H \rangle := w_{\sltwo} ((I, H))$; \\
    $\overline{I}$ & the dual share of the share $I$, i.e., the share obtained by switching the orientation of one of the strands of $I$; \\
    $\overline{\Gamma}$ & the complement graph of a graph $\Gamma$, i.e., the graph with the same vertex set and the complementary edge set; \\
    $U$ & the operator of adding to a share an arch that intersects all the bridges and does not intersect any arch of a given share (in \cite{Zakorko} and \cite{KaZi} this operator is denoted by $S$); \\ 
    $X$ & the operator of adding to a share a bridge that intersects every bridge and does not intersect any arch of a given share; (cf. with operator $T$ in \cite{Zakorko}); \\
    $Y$ & the operator of adding to a share a bridge that does not intersect any other chord of a given share; \\
    $x^m$ & a basis in $\Shares$, collections of $m$ pairwise non-intersectiong bridges; \\
    $y^m$ & a basis in $\Shares$, collections of $m$ pairwise intersecting bridges; \\
    $e_m$ & a basis in $\Shares$, the eigenbasis of the operator $U$; \\
    $p_m$ & a basis in $\Shares$, orthogonal w.r.t. $\langle \cdot, \cdot \rangle $ (the basis was denoted by $y_m$ in \cite{Zakorko}); \\
    $u_{i,m}$ & coefficients of $U$ w.r.t. the basis $y^n$, $n = 0,1,\ldots$; \\
    $u_m$ & eigenvalues of $U$; $u_m = u_{m,m}$; \\
    $G_{I}$ & the generating function $\sum_{k=0}^{\infty} \langle I,y^k \rangle t^k$ for $I\in \Shares$; \\
    $\CB_m$ & the generating function of the values of $w_{\sltwo}$ on complete bipartite graphs (or on chord diagrams $(y^m,y^n)$ for $n = 0,1,2,\ldots$) (this generating function is denoted by $G_m$ in \cite{KaZi}); \\
    $\Split_m$ & the generating function of the values of $w_{\sltwo}$ on split graphs (or on chord diagrams $(x^m,y^n)$ for $n = 0,1,2,\ldots$); \\
    $r_k^{\Gamma}$ & the coefficient of $k$th geometric series in $G_{\Gamma}$. 
\end{tabular}

\section{Chord diagrams} 
\label{sec:chord_diagrams}

In this section we recall the main notions of the theory
of weight systems we require in the present paper.
The reader may also find useful to consult~\cite{Chmutov_Duzhin_Mostovoy} or~\cite{LZ}.

\subsection{Chord diagrams and intersection graphs}

A \textbf{chord diagram} of order $n$ is an oriented circle with $2n$ pairwise distinct points on it split into $n$ pairs, considered up to orientation-preserving diffeomorphisms of the circle.  
For convenience, in the pictures below we connect points from one pair with a \textbf{chord}, which is
shown either as a segment 
or as an arc lying inside the circle. 

An \textbf{arc diagram} of order $n$ is an oriented line, which we call a \textbf{strand} with $2n$ pairwise distinct points on it split into $n$ pairs, considered up to orientation-preserving diffeomorphisms of the line. 
As in the definition of chord diagrams, we connect points from one pair with an arc, lying in a fixed half-plane.
We can obtain an arc diagram from a chord diagram by cutting a circle at some point different from  the $2n$ endpoints of diagram chords. A chord diagram is uniquely reconstructred from
an arc diagram, while a given chord diagram of order~$n$
can have up to $2n$ distinct presentations as an arc diagram.

We say that two chords \textbf{intersect} if their ends alternate.
The \textbf{intersection graph} of a chord diagram is the graph such that its vertices correspond one-to-one to the chords of the diagram, and two vertices are adjacent if and only if the corresponding chords intersect one another. 

\begin{figure}[h!]
    \centering
    \begin{tabular}{ c c }
        \raisebox{-28pt}{\includegraphics[width=50pt]{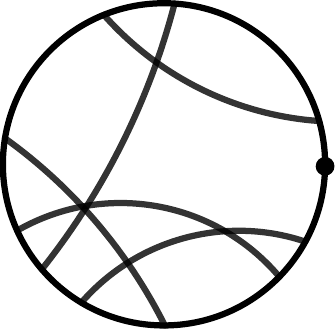}} \qquad
        \raisebox{-28pt}{\includegraphics[width=50pt]{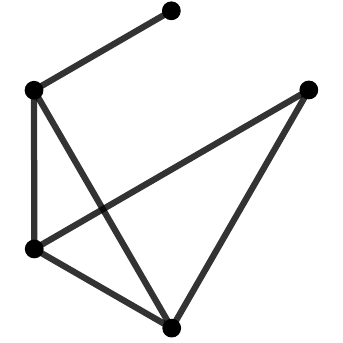}} \quad 
        \raisebox{-28pt}{\includegraphics[width=100pt]{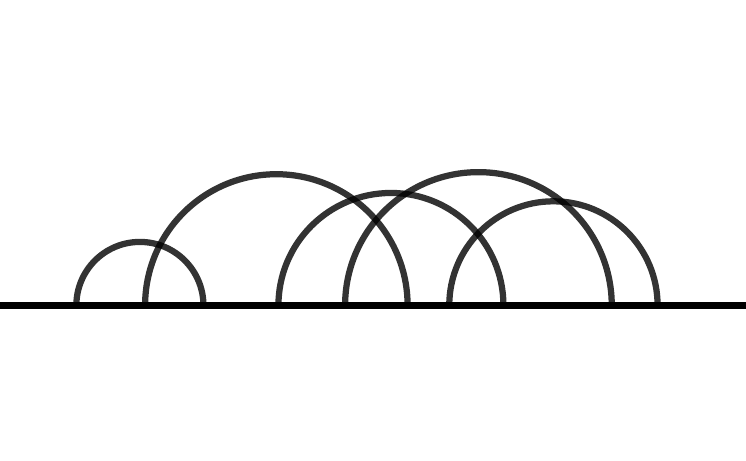}}
    \end{tabular}
    \caption{A chord diagram with a cutting point and the corresponding intersection graph and arc diagram}
    \label{fig:three_pictures}
\end{figure}

Notice that not each simple graph is the intersection graph
of some chord diagram. 
For example, every graph in Figure~\ref{fig:bouchet2} is not an intersection graph.
Moreover, in terms of these three graphs Bouchet described all the graphs that are not intersection graphs. 
For this description, consider the operation on graphs,
which replaces the subgraph induced by the neighbourhood
of a given vertex by its complement. (The {\bf neighborhood} of a vertex~$v$ is the
 the set of vertices connected with~$v$).
Let us call two graphs \textbf{locally equivalent} if we can obtain one of them from the other one by a sequence of such local operations.

\begin{Claim}[\cite{Bouchet}]
\label{claim:bouchet}
A graph~$\Gamma$ is not the intersection graph of any chord diagram if and only if there exists a graph locally equivalent 
to~$\Gamma$, which contains as a subgraph at least one of the graphs in Fig.~\ref{fig:bouchet2}.
\begin{figure}[h!]  
    \begin{center}
        \includegraphics[width=200pt]{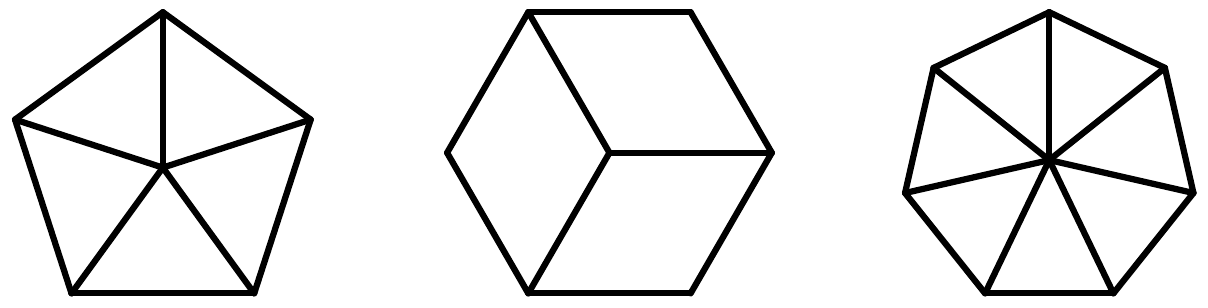}.
        \caption{Graphs that are not intersection graphs}
        \label{fig:bouchet2}
    \end{center}
\end{figure}    
\end{Claim}

\subsection{The $\sltwo$ weight system on chord diagrams}
Linear combinations of chord diagrams with coefficients in $\mathbb{C}$ form the 
\textbf{vector space of chord diagrams}, which we denote by $C$.
The \textbf{four-term elements}
\begin{equation*}
    \raisebox{-13pt}{\includegraphics[width=30pt]{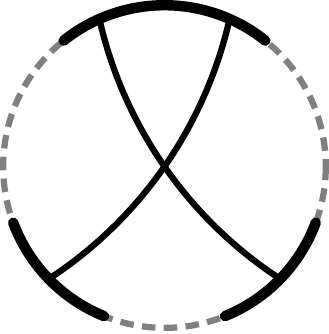}} - 
    \raisebox{-13pt}{\includegraphics[width=30pt]{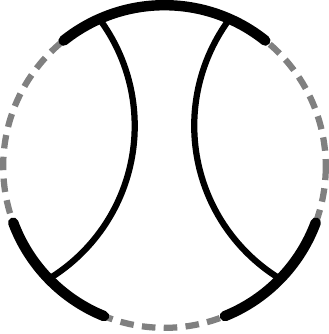}} -
    \raisebox{-13pt}{\includegraphics[width=30pt]{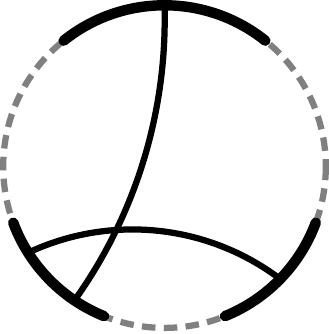}} +
    \raisebox{-13pt}{\includegraphics[width=30pt]{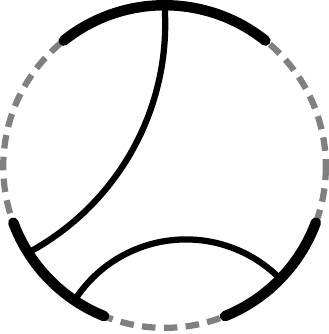}}
\end{equation*}
span a subspace in~$C$.
In this expression, as well as in all similar pictures below, the diagrams can have chords with endpoints on the dashed arcs of the 
circle, and these additional chords are arranged in the same way in all the four diagrams.
In the quotient space of~$C$ modulo four-term elements, one can define a product. 
For two chord diagrams $D_1$ and $D_2$, the chord diagram
associated to the result of concatenating their arc diagrams is called their {\bf product} and denoted by $D_1 \cdot D_2$:
\begin{equation*}
    \raisebox{-13pt}{\includegraphics[width=30pt]{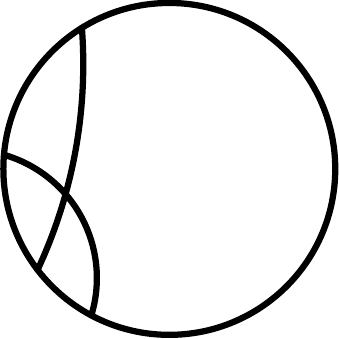}}\cdot
    \raisebox{-13pt}{\includegraphics[width=30pt]{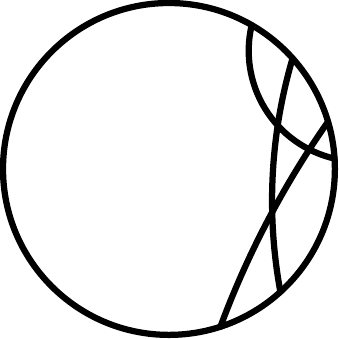}} =
    \raisebox{-13pt}{\includegraphics[width=30pt]{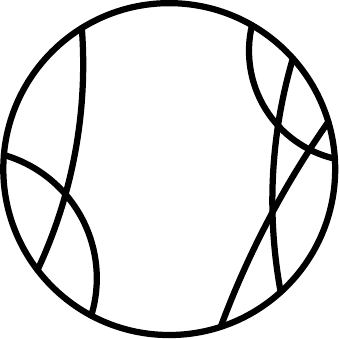}}.
\end{equation*}
The product we obtain depends on the choice of
the cutting points on the factors, whence it is not well-defined for $C$. However, all the products for all the pairs of arc diagrams are the same modulo four-term elements, which gives us a multiplication on the quotient space.

A \textbf{weight system} is a linear function $w$ on $C$ that vanishes on every four-term element:
\begin{equation*}
    w(\raisebox{-13pt}{\includegraphics[width=30pt]{pic/4_term1.pdf}} - 
    \raisebox{-13pt}{\includegraphics[width=30pt]{pic/4_term2.pdf}} -
    \raisebox{-13pt}{\includegraphics[width=30pt]{pic/4_term3.pdf}} +
    \raisebox{-13pt}{\includegraphics[width=30pt]{pic/4_term4.pdf}}) = 0.
\end{equation*}
This equation is called \textbf{the four-term relation}.
A weight system is said to be \textbf{multiplicative}
if it takes the product of two chord diagrams to the
product of its values on the factors.
From now on we omit the function $w$ in the diagram equations identifying a diagram with its value.

For a Lie algebra $\mathfrak{g}$ of dimension~$d$
endowed with a nondegenerate invariant bilinear form, one can construct a multiplicative weight system $w_\mathfrak{g}$.
To this end, we choose any orthonormal basis $x_1, x_2, \ldots, x_d$ with respect to the bilinear form.
First let us construct a function $w$ on the vector space of arc diagrams taking values in the universal enveloping algebra of $\mathfrak{g}$, which we denote by $U(\mathfrak{g})$.
For every map from the set of arcs of a given diagram to the set $\{1,2, \dots d\}$, we assign the basis element $x_i$ to both ends of an arc taken to~$i$ and take the product of all these elements along the strand.
The sum of these products over all the mappings gives us the image of the arc diagram in $U(\mathfrak{g})$, and we extend $w_\mathfrak{g}$ to linear combinations of arc diagrams
by linearity.
For example, the  value of $w_\mathfrak{g}$ on the arc diagram 
in Fig.~\ref{fig:three_pictures} is
\begin{equation*}
    \sum_{i_1 = 1}^d \sum_{i_2 = 1}^d \sum_{i_3 = 1}^d
    \sum_{i_4 = 1}^d \sum_{i_5 = 1}^d x_{i_1} x_{i_2} x_{i_1} x_{i_3} x_{i_4} x_{i_2} x_{i_5} x_{i_3} x_{i_4} x_{i_5}.
\end{equation*}
Instead of taking an orthonormal basis, one can also take 
two mutually dual bases with respect to the given bilinear form (see \cite{Chmutov_Duzhin_Mostovoy}).

In order to check that $w_\mathfrak{g}$ is well defined on the space of chord diagrams, one has to verify that $w_\mathfrak{g}$ takes the same value on different arc
presentations of a chord diagram.
This is guaranteed by the following assertion.
\begin{Claim}[{\cite{Ko}}]
The following assertions are true:
    \begin{enumerate}
        \item $w_\mathfrak{g}$ is independent of the choice of the orthonormal basis $x_1, x_2, \ldots, x_d$;
        \item $w_\mathfrak{g}$ takes the same value on any two presentations of a chord diagram;
         \item the image of $w_\mathfrak{g}$ lies in the center of $U(\mathfrak{g})$;
           \item $w_\mathfrak{g}$ satisfies the four-term relation.    
    \end{enumerate}
\end{Claim}
In the present paper, we deal with the weight system
associated to the simplest nontrivial Lie algebra 
$\mathfrak{g} = \sltwo$ endowed with the 
bilinear form $(x,y)=2\Tr(xy)=\frac{1}{2}B(x,y)$, where $B$ is the standard Killing form.
In this case we choose an orthonormal basis 
\begin{equation}
    \label{eq:x1x2x3}
    x_1 = \frac{1}{2}\begin{pmatrix}
        0 & 1 \\
        1 & 0
    \end{pmatrix}, \quad
    x_2 = \frac{1}{2} \begin{pmatrix}
        0 & -i \\
        i & 0
    \end{pmatrix}, \quad
    x_3 = \frac{1}{2} \begin{pmatrix}
        1 & 0 \\
        0 & -1
    \end{pmatrix},
\end{equation} 
for which $[x_i, x_j] = i \varepsilon_{ijk} x_k$
holds, where $\varepsilon_{ijk}$ is the Levi-Civita symbol.
The center of the universal enveloping algebra of $\mathfrak{g}$ is generated by the Casimir element $c:=x_1^2+x_2^2+x_3^2$.
The resulting weight system is denoted by $w_{\sltwo}$, called $\sltwo$ \textbf{weight system}, and takes values in $\mathbb{C}[c]$.
On the simplest diagrams, with zero and one chord, it takes the following values:
\begin{equation} \label{eq:wsl2_initial}
    \raisebox{-13pt}{\includegraphics[width=30pt]{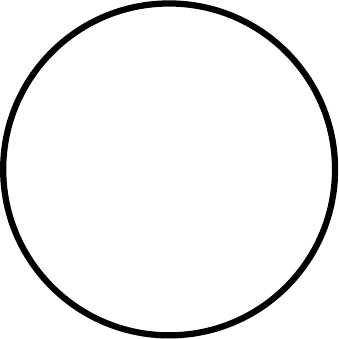}} = 1, \qquad 
    \raisebox{-13pt}{\includegraphics[width=30pt]{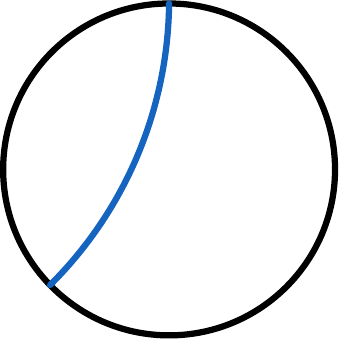}} = c. 
\end{equation}
The following combinatorial relations serve as a main tool for computing the values of $w_{\sltwo}$. 
\begin{Claim}[Chmutov--Varchenko relations, \cite{ChV}]
    Let $D$ be a chord diagram of order $n\ge2$
    with a connected intersection graph. 
    Then there are the following mutually exclusive and exhaustive cases.
    \begin{enumerate}
    \item The diagram $D$ contains a \textbf{leaf}, which is a chord that intersects precisely one chord. 
    Then we have
        
    \begin{figure}[h!]
        \begin{equation*}
            \raisebox{-13pt}{\includegraphics[width=30pt]{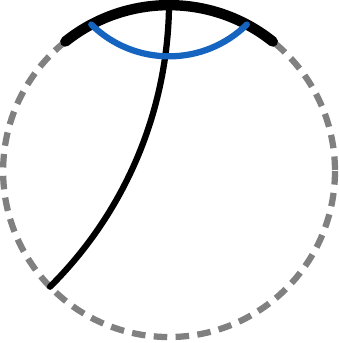}} = (c-1) \cdot 
            \raisebox{-13pt}{\includegraphics[width=30pt]{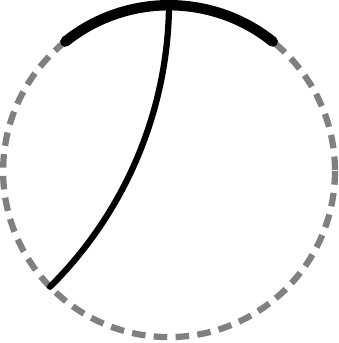}}
        \end{equation*}
        \begin{equation*}
            w_{\sltwo}(D) = (c-1) w_{\sltwo}(D'),
        \end{equation*}
    \end{figure}
    where $D'$ is the diagram $D$ with the leaf removed.

    \item The chord diagram $D$ contains no leaves, then it contains three chords in one of the leftmost configurations in the two equations shown in~Fig.~\ref{fig:6term}, and the equations themselves hold.
    \begin{figure}[h!]
        \begin{equation*}
            \raisebox{-13pt}{\includegraphics[width=30pt]{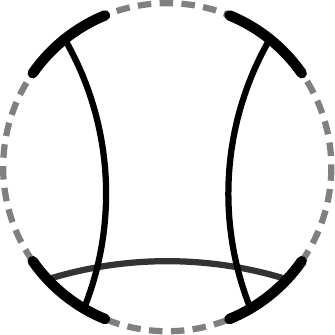}} =
            \raisebox{-13pt}{\includegraphics[width=30pt]{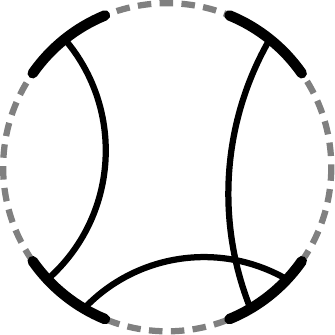}} +
            \raisebox{-13pt}{\includegraphics[width=30pt]{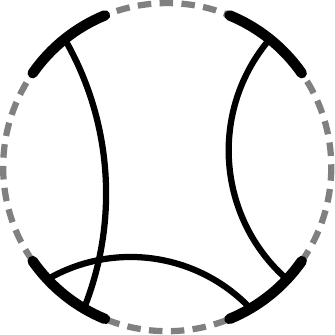}} - 
            \raisebox{-13pt}{\includegraphics[width=30pt]{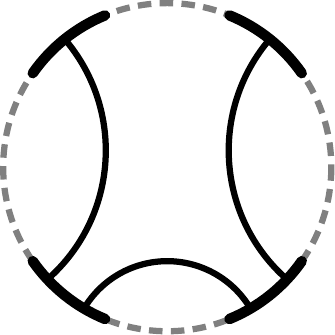}} +
            \raisebox{-13pt}{\includegraphics[width=30pt]{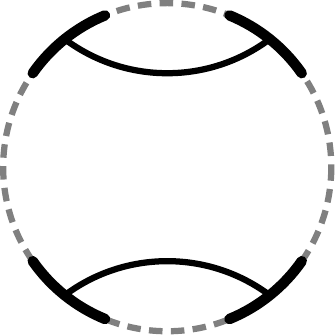}} -
            \raisebox{-13pt}{\includegraphics[width=30pt]{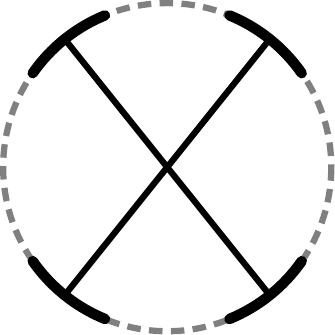}};
        \end{equation*}
        \begin{equation*}
            \raisebox{-13pt}{\includegraphics[width=30pt]{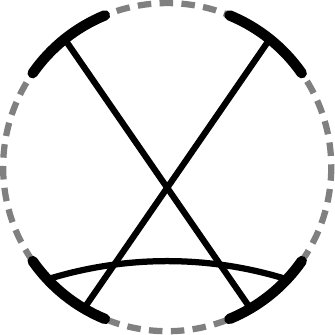}} =
            \raisebox{-13pt}{\includegraphics[width=30pt]{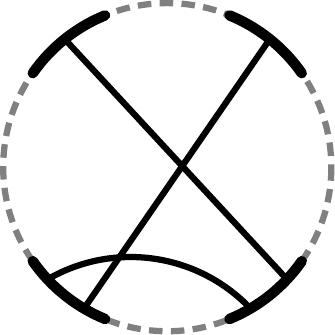}} +
            \raisebox{-13pt}{\includegraphics[width=30pt]{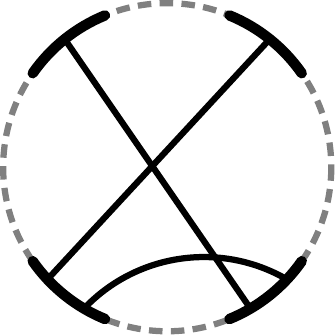}} - 
            \raisebox{-13pt}{\includegraphics[width=30pt]{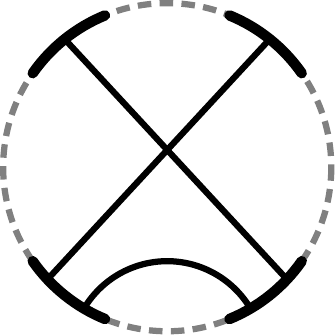}} +
            \raisebox{-13pt}{\includegraphics[width=30pt]{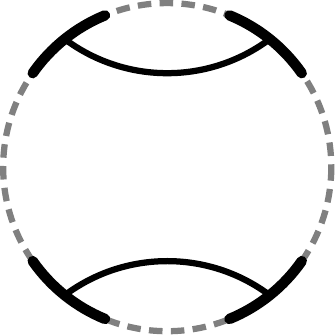}} -
            \raisebox{-13pt}{\includegraphics[width=30pt]{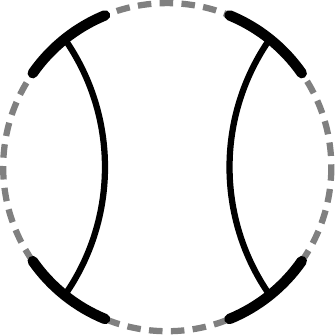}}.
        \end{equation*}
        \caption{Chmutov--Varchenko $6$-term relations
        for $\sltwo$ weight system values}
        \label{fig:6term}
    \end{figure}  
    \end{enumerate}
\end{Claim}
The values of $w_{\sltwo}$ on every chord diagram can be computed using only the initial values \eqref{eq:wsl2_initial}, the multiplicativity of $w_{\sltwo}$, and the Chmutov-Varchenko relations. 
For example, on the chord diagram with $n$ chords such that its intersection graph is a tree, $w_{\sltwo}$ takes the value $c(c-1)^{n-1}$.
\begin{Claim}[\cite{ChL}]
    \label{claim:ChL}
    The $\sltwo$ weight system depends only on the intersection graph of a chord diagram.
\end{Claim}

For any graph $\Gamma$, let us denote its vertex set by $V(\Gamma)$.
Let $A,B$ be two vertices of a simple graph $\Gamma$. By $\Gamma'_{AB}$ denote the graph obtained from $\Gamma$ by changing the adjacency between the vertices $A$ and $B$ in $\Gamma$, that is, by erasing the edge $AB$ in the case this edge exists and by adding the edge otherwise. By $\tilde \Gamma_{AB}$ denote the graph obtained from $\Gamma$ as follows. For any vertex $C$ in $V(\Gamma)\setminus\{A,B\}$ we change its adjacency with $A$ if $C$ is joined to $B$ and do nothing otherwise. 
A \textbf{four-term element in the space of graphs} is a linear combination
\begin{equation}
    \label{eq:4term_graphs}
    \Gamma - \Gamma'_{AB} - \tilde \Gamma_{AB} + \tilde \Gamma_{AB}'.
\end{equation}

Note that the linear combination of intersection graphs of the summands in a four-term element in $C$ gives exactly a four-term element in the space of graphs.
The following problem has been stated about 15 years ago.
\begin{question*}[S.~Lando]
    Does there exist a graph invariant satisfying the four-term relations that coincides with the $\sltwo$-weight system on the intersection graphs?
\end{question*}

The problem still remains open, and one of the main goals
of the present paper consists in supplying data that may
help to answer this question either in the affirmative
or in the negative.

\section{$\sltwo$ weight system of chord diagrams on two strands}
\subsection{Chord diagrams on $k$ strands}
\label{sec:chord_diagrams_k}

A \textbf{chord diagram on $k$ strands} is an ordered set of oriented lines, called strands, with $2n$ pairwise distinct points on them split into $n$ pairs, considered up to orientation-preserving diffeomorphisms of each strand. 
We connect the points in a pair by a curve which we call a \textbf{chord}. 
If both points lie on the same strand, then we say that this chord is an \textbf{arch}
(note the difference with the notion `arc' which we preserve for arcs in arc diagrams). 
If the points lie on different strands, then we say that this chord is a \textbf{bridge}. 

Denote by $A_k$ the vector space spanned by chord diagrams on $k$ strands with coefficients in $\mathbb{C}$.
We endow $A_k$ with the multiplication, which is concatenation of two diagrams with matched orientation of the strands.

We can define a \textbf{weight system} on $A_k$ as a linear function that satisfies the \textbf{4-term relation} shown on Figure~\ref{fig:4_term_cd_k}.
One can obtain the 4-term relation for chord diagrams from this generalization by an orientation-preserving embedding of strands into a circle.

\begin{figure}[h!]
    \begin{equation*}
         \raisebox{-18pt}{\includegraphics[width=30pt]{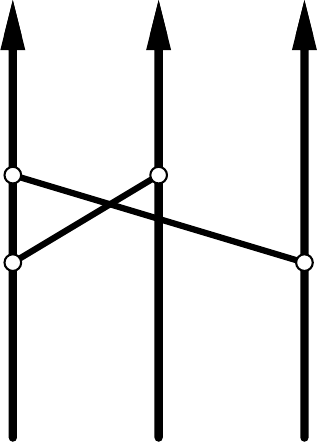}} - 
         \raisebox{-18pt}{\includegraphics[width=30pt]{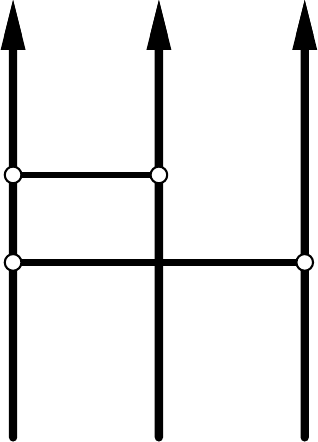}} = 
         \raisebox{-18pt}{\includegraphics[width=30pt]{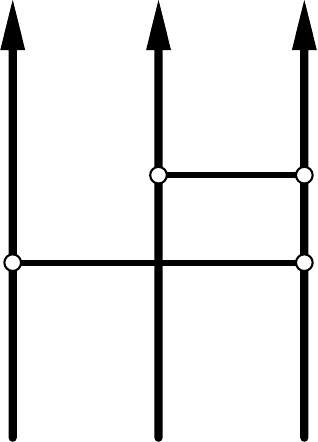}} - 
         \raisebox{-18pt}{\includegraphics[width=30pt]{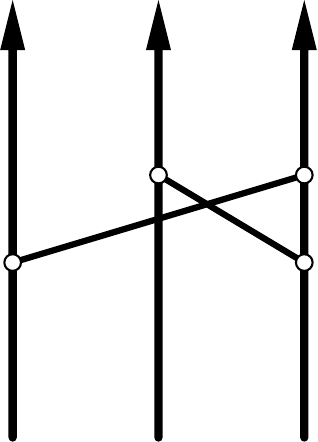}}.
    \end{equation*}
    \caption{The 4-term relation on chord diagrams on $k$ strands. 
    Each arrow represents a part of a strand. 
    Any two of these parts may belong to the same strand.}
    \label{fig:4_term_cd_k}
\end{figure}

The vector space $A_k$ can be endowed with a structure of a non-commutative algebra with respect
to the dot multiplication $\cdot\colon A_k\times A_k \to A_k$, which concatenates two chord diagrams on $k$ strands, see Fig.~\ref{fig:share_algebra_dot} 

\begin{figure}[h!]
    \centering
    \begin{equation*}
        \raisebox{-30pt}{\includegraphics[width=40pt]{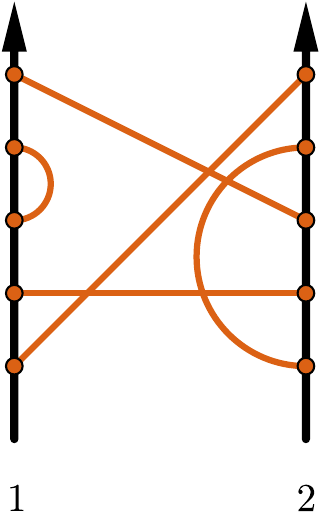}} \cdot 
        \raisebox{-30pt}{\includegraphics[width=40pt]{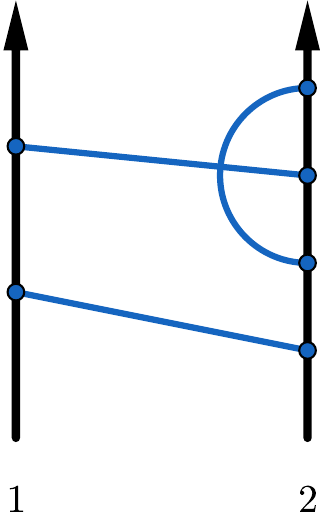}} = 
        \raisebox{-30pt}{\includegraphics[width=40pt]{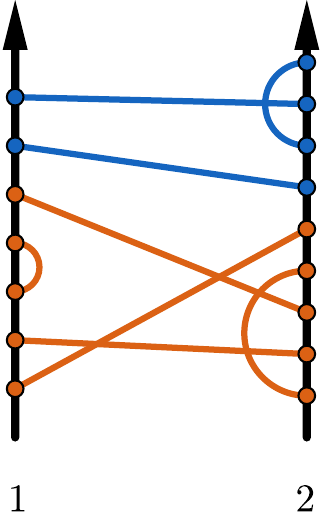}}.
    \end{equation*}
    \caption{Dot product of two chord diagrams on $2$ strands}
    \label{fig:share_algebra_dot}
\end{figure}

Similarly to the case of chord diagrams, using a Lie algebra $\mathfrak{g}$ (again, endowed with a nondegenerate invariant bilinear form) one can construct a weight system $w_{\mathfrak{g}}\colon A_k \to U(\mathfrak{g})^{\otimes k}$,
which takes values in the $k$~th tensor power of the universal enveloping algebra of~${\mathfrak g}$. 
The construction follows the one for chord diagrams.
Again, we choose an orthonormal basis $x_1, x_2, \dots, x_d$ with respect to the bilinear form. 
For every map from the set of chords of a given diagram on $k$ strands to the set $\{1, 2, \dots d\}$, we assign the basis element $x_i$ to both ends of a chord taken to $i$, then we take the product of all these elements along every strand, and finally take the tensor product of these products according to the order of the strands. 
The sum of the tensor products over all the mappings gives us the image of the chord diagram on strands in $U(\mathfrak{g})^{\otimes k}$, and we extend $w_{\mathfrak{g}}$ to linear combinations of diagrams by linearity.
This construction was discussed in \cite{DuzhKar} for the  Lie algebra
$\mathfrak{gl}_N$.

\subsection{$\sltwo$ weight system on chord diagrams on two strands}

From now on, we mostly consider chord diagrams on $2$ strands. 

Similarly to the case of arc diagrams, we can obtain a chord diagram on two strands from a usual chord diagram by cutting the outer circle at two points none of which is an end of a chord.
We can also consider a share as a part of a chord diagram
formed by a subset of chords all whose ends belong to two given segments
of the outer circle, not containing ends of other chords.
Then the chords of the diagram that are not included in the share form the \textbf{complement share}. 
Conversely, a chord diagram can be obtained from a share as a \textbf{closure of the share}, which means that we glue the ends of the strands in such a way that their orientations agree.

We say that two chords of the share intersect, if any one of the following conditions holds:
\begin{itemize}
    \item both chords are arches with the alternating points on one strand;
    \item one chord is an arch and the other one is a bridge with an endpoint between the ends of the arch;
    \item both chords are bridges, and their ends are arranged in a different order on the strands.
\end{itemize}

\begin{figure}[h!]
    \centering
    \includegraphics[width=40pt]{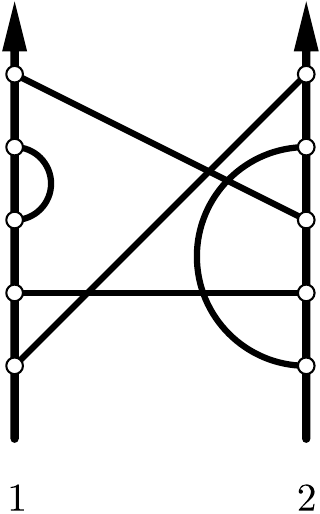} \qquad \raisebox{10pt}{\includegraphics[width=50pt]{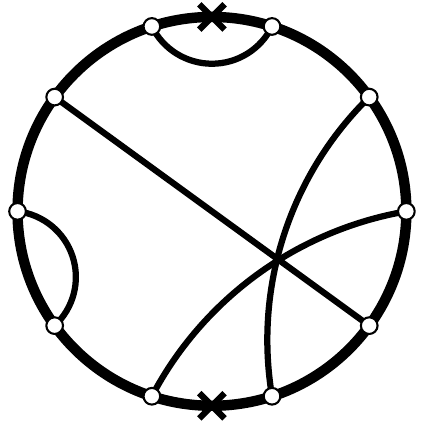}}
    \caption{A share with two arches and three bridges and its closure}
    \label{fig:share_arc_brdg}
\end{figure}

In addition to the dot multiplication,
we have also another way to define a structure of an algebra on $A_2$.
It corresponds to 
the different multiplication $\times\colon A_2\times A_2 \to A_2$
shown in Fig.~\ref{fig:share_algebra_cross}.
Both algebras $(A_2, \cdot)$ and $(A_2, \times)$ are unital with identity element $\mathds{1}$, which is the empty share.

\begin{figure}[h!]
    \centering
    \begin{equation*}
        \raisebox{-30pt}{\includegraphics[width=40pt]{pic/share_oriented_colored1.pdf}} \times 
        \raisebox{-30pt}{\includegraphics[width=40pt]{pic/share_oriented_colored2.pdf}} = 
        \raisebox{-30pt}{\includegraphics[width=40pt]{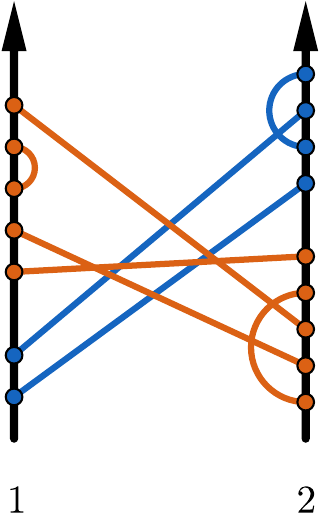}}.
    \end{equation*}
    \caption{Cross-multiplication of two shares}
    \label{fig:share_algebra_cross}
\end{figure}

One can construct a chord diagram 
from two shares $I$ and $H$ by
treating them as being
complement to each other. 
Denote this chord diagram by $(I, H)$ if the first (and second) strand of $I$ is followed by the first 
(respectively, the second) strand of $H$.
In other words, $(I, H)$ is the closure of the
dot-product of $I$ and $H$.
The chord diagram $(I,\mathds{1})$ is the closure of $I$.
 Notation $(I, H)$ extends to $A_2$ by linearity:
\begin{equation*}
    (I, H) := \sum_{ij} \alpha_i \beta_j (I_i, H_j) \in C,
\end{equation*}
where $I = \sum_{i} \alpha_i I_i$ and $H = \sum_{j} \beta_j H_j$ are linear combinations of shares $I_i$ and $H_j$ with coefficients $\alpha_i, \beta_j\in \mathbb{C}$.

For the weight system $w_{\sltwo}$ endowed with bilinear form $(\xi_1,\xi_2)=2\Tr(\xi_1\xi_2)$, the sum $x_1\otimes x_1 + x_2 \otimes x_2 + x_3 \otimes x_3$ corresponds to a diagram with one bridge (recall that the orthonormal basis $\{x_1,x_2,x_3\}$ was described in \eqref{eq:x1x2x3}).
As in the case of chord diagrams, we have the Chmutov-Varchenko relations.
Here we should be careful with the signs, because of the way we choose the orientation on the strands (compare the signs of the last two diagrams in Figure~\ref{fig:6term} and those in Figure~\ref{fig:share_6term}).
\begin{figure}[h!]
\begin{equation*}
    \raisebox{-17pt}{\includegraphics[width=30pt]{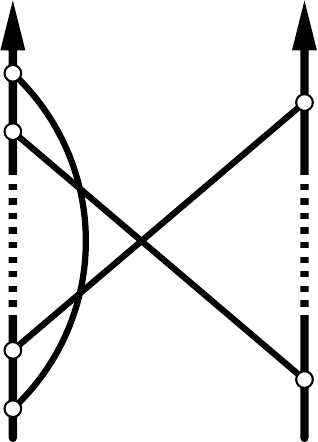}} =
    \raisebox{-17pt}{\includegraphics[width=30pt]{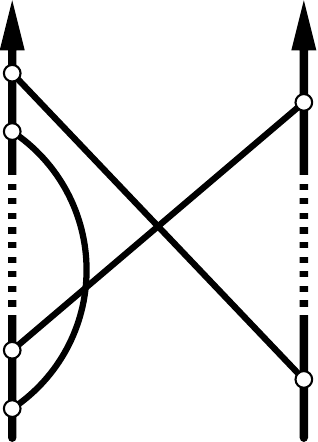}} +
    \raisebox{-17pt}{\includegraphics[width=30pt]{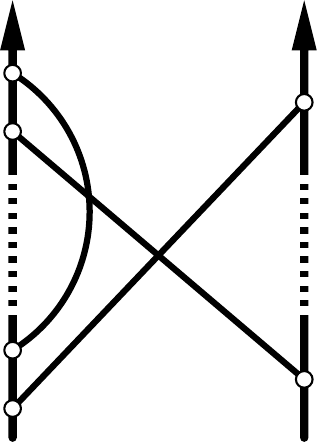}} - 
    \raisebox{-17pt}{\includegraphics[width=30pt]{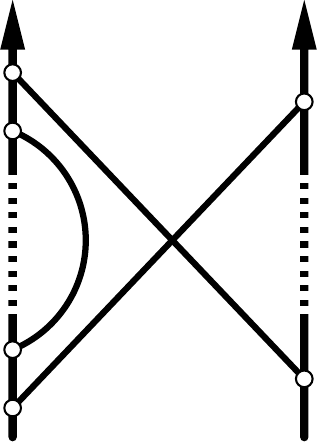}} -
    \raisebox{-17pt}{\includegraphics[width=30pt]{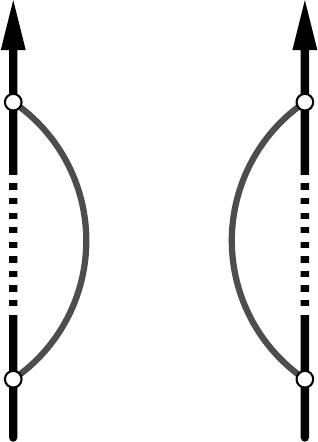}} +
    \raisebox{-17pt}{\includegraphics[width=30pt]{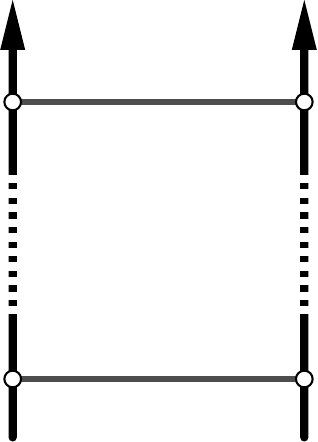}}.
\end{equation*}
    \caption{A six-term relations on shares. 
    Note the signs of the last two diagrams}
    \label{fig:share_6term}
\end{figure}

Note that $w_{\sltwo}$ is multiplicative if we consider $A_2$ as an algebra with dot product, therefore, $w_{\sltwo}$ is a homomorphism of algebras. 
Consider the quotient algebra $A_2 / \Ker w_{\sltwo}$, which is isomorphic to $\Ima w_{\sltwo}$.
It turns out that this quotient algebra is isomorphic to the ring of polynomials in three variables.
\begin{Theorem} We have
    \label{thm:sltwoA2}
    $\Ima w_{\sltwo} = \mathbb{C}[c_1, c_2, x]\subset U(\sltwo)$, where 
    \begin{equation*}
        c_1 := c\otimes 1, \quad c_2 := 1\otimes c, \quad 
        x := x_1\otimes x_1 + x_2 \otimes x_2 + x_3 \otimes x_3.
    \end{equation*}
\end{Theorem}
\begin{proof}
    First, we show that the value of $w_{\sltwo}$ on every chord diagram on two strands is a polynomial in $x$, $c_1$ and $c_2$.
    The proof of algebraic independence of $c_1$, $c_2$ and $x$ is given in the Appendix.

    In order to prove the first assertion, 
    we construct for every element $I\in A_2$ its \textbf{normal form} $I_{norm}\in A_2$, which is a polynomial of chord diagrams on two strands with no intersecting chords such that $w_{\sltwo}(I) = w_{\sltwo}(I_{norm})$. 
    Due to the multiplicativity of $w_{\sltwo}$ with respect to the dot product, the value of the $\sltwo$-weight system on a share without intersecting chords that has $k$ arches on the right-hand side strand, $m$ arches on the left-hand side strand, and $n$ bridges, is equal to $c_1^k c_2^m x^n$, thus, $w_{\sltwo}(I_{norm})$ is indeed a polynomial in these variables.
    
    Let us say that a chord diagram on two strands $I'$ is \textbf{simpler} than $I$ if one of the following conditions is satisfied:
    \begin{enumerate}
        \item $I'$ has less chords than $I$;
        \item $I$ and $I'$ have the same number of chords, $I'$ has arches, but $I$ does not have any arch;
        \item $I$ and $I'$ have the same number of chords, both $I$ and $I'$ have arches, but the minimal \textbf{length of an arch} in $I'$ is less than that for $I$, where the length of an arch is the number of chord endpoints lying between the endpoints of an arch;
        \item $I$ and $I'$ have the same number of chords and don't have arches, but the number of intersections of chords in $I'$ is smaller than that in~$I$.
    \end{enumerate}
    The construction of a normal form proceeds as follows: 
    \begin{enumerate}
        \item If $I$ has an arch of length $0$ or $1$, then it can be simplified using multiplicativity of the weight system or leaf-removing relation;
        \item If $I$ has only arches of lengths more than $1$, then $I$ can be simplified with a suitable six-term relation;
        \item If $I$ has no arches, but it contains a pair of intersecting bridges, then $I$ can be simplified using a four-term relation;
        \item Otherwise, $I$ has no arches, and its bridges do not intersect, hence it is already in the normal form.
    \end{enumerate}
\end{proof}
From now on we will refer to the values of the $\sltwo$ weight system on the shares as to polynomials in $x$, $c_1$ and $c_2$.

\subsection{Algebra $\Shares$}
There is another way to construct a quotient algebra 
starting from the weight system $\sltwo$. 
Let us call two elements $I_1, I_2 \in A_2$ \textbf{equivalent} if for every complement share $H$ the values of the weight system $\sltwo$ on the linear combinations of  chord diagrams $(I_1, H)$ and $(I_2, H)$ are the same.
For two equivalent elements, we write $I_1 \sim I_2$.
It turns out that 
the quotient algebra $\Shares := A_2/{\sim}$ is isomorphic to the algebra of polynomials in two variables.
\begin{Claim}[cf.~\cite{Zakorko}]
    \label{claim:two_definitions}
    Two elements $I_1, I_2 \in A_2$ are equivalent if and only if the values of the $w_{\sltwo}$ on these two elements become equal after the substitution $c_1=c_2 = c$, i.e., 
    $$w_{\sltwo}(I_1)|_{c1=c=c_2} = w_{\sltwo}(I_2)|_{c1=c=c_2}.$$
\end{Claim}

\begin{proof}
    In one direction the statements follows directly from the definition of $w_{\sltwo}$, so we focus on proving that equivalence of $I_1$ and $I_2$ implies the equality of their values.
    
    Suppose we have a non-zero normal form $I_{norm}$, for which $w_{\sltwo}((I_{norm}, H))=0$ for all $H\in S$.
    This vanishing property of $w_{\sltwo}$ on chord diagrams implies that there are  
    coefficients $\alpha_k(c)$ such that for all $n$
    \begin{equation}
        \label{eq:complete_gen_function}
        \sum_{k=0}^{m} \alpha_k(c)\cdot w_{\sltwo}((x^k, x^n)) = 0,
    \end{equation}
    where $m$ is a number independent of $n$.
    The intersection graph of the chord diagram $(x^k, x^n)$ is the complete graph $K_{k+n}$, and we know that the generating function for $w_{\sltwo} (K_{k+n})$, $n=0,1,2,\dots$, has the form of
   an infinite continued fraction (see \cite{Zakorko}).
    In particular, it is not a rational function.
    However, according to \eqref{eq:complete_gen_function} the generating function should be a rational function, and we arrive at a contradiction.
\end{proof}
\begin{corollary}
    The quotient algebra $\Shares=A_2/{\sim}$ is isomorphic to the polynomial algebra $\mathbb C[c, x]$, and isomorphism comes from a homomorphism of algebras $\psi\colon\Shares \to \mathbb C[c, x]$, which sends
    the share whose only chord is a bridge to $x$.
    That gives us an isomorphism $\Shares\cong \mathbb{C}[c, x]$.
\end{corollary}

From now on we study the quotient algebra $\Shares$. 
The sequence of shares $\mathds 1, x, x^2, x^3,\dots$ forms a free basis in $\Shares$ viewed as a module over $\mathbb C[c]$. 
Another free basis in the $\mathbb C[c]$-module $\Shares$ is $\mathds 1, y, y^2, y^3, \dots$ for $y^k = x^{\times k}$.
\begin{figure}[h!]
    \centering
    \includegraphics[width=40pt]{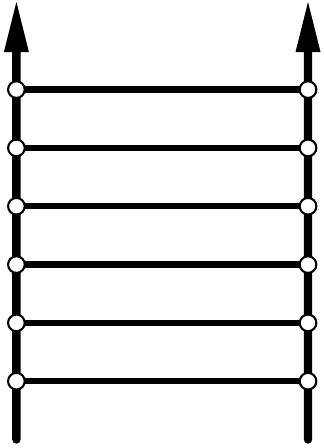} \qquad \includegraphics[width=40pt]{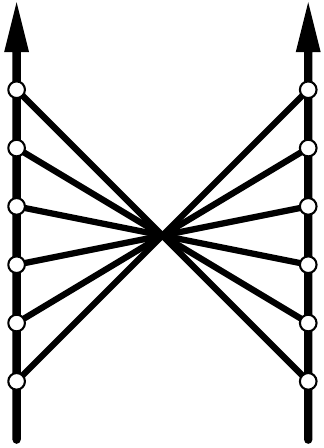}
    \caption{Shares $x^6$ and $y^6$}
    \label{fig:enter-label}
\end{figure}
\begin{corollary}[cf.~\cite{Zakorko}]
    The algebra $(\Shares, \times)$ is isomorphic to $\mathbb C[c, y]$, under the isomorphism sending
    the share with one bridge to $y$.
\end{corollary}

From now on we treat $\Shares$ as an algebra with cross-multiplication. 

The Chmutov--Varchenko relations are not homogeneous, hence the number of chords does not induce any grading of $\Shares$. 
Nevertheless, there is a natural filtration
\begin{equation}
\Shares_0 \subset \Shares_1 \subset \Shares_2  \subset \ldots \subset \Shares,
\label{eq:filtration}
\end{equation}
where $\Shares_m$ is spanned over $\mathbb C[c]$ by all the shares $y^k$ with $k\leqslant m$.
For this filtration, the following assertions hold.
\begin{Lemma}
    \label{lem:bridges_filtration}
    If a share $I$ contains exactly $m$ bridges, then $I \in \Shares_m$.
\end{Lemma}
\begin{proof}
     As elements of the quotient algebra, $I = I_{norm}$. 
     Each summand of the normal form $I_{norm}$ has at most $m$ bridges, because each step of the algorithm described in the proof of Theorem~\ref{thm:sltwoA2} does not increase the number of bridges.
     Therefore, $I$ is indeed in $\Shares_m$.
\end{proof}

The following lemma was proven in \cite{Zakorko}.
\begin{Lemma}
    \label{lemma:xy}
    The normal form for a share $y^m$ is of the form $x^m + O(x^{m-1})$, where we denote by $O(y^{n-1})$
a linear combination of summands lying in $\Shares_{n-1}$.
\end{Lemma}

Now we can obtain more general similar lemma.
\begin{Lemma}
    \label{corol:3:x^m+...}
    If a share $I$ has $m$ bridges and no arches, then it is equal to a monic polynomial in $x$ (or $y$) of degree $m$.
\end{Lemma}
\begin{proof}
    We can prove this lemma by induction on the number of intersections of the chords.
    If $I$ has no intersecting bridges, then it is equal to $x^m$, and the statement holds.
    For a share $I$ with non-zero number of intersecting bridges (and without arches) we can apply a four-term relation.
    In this way we replace a single share $I$ with a linear combination of three shares,
    one of which satisfies the induction hypothesis,
    while two others belong to $\Shares_{m-1}$. 
    
    For the basis $y^m$ it remains to apply Lemma~\ref{lemma:xy}.
\end{proof}

As another corollary of Theorem~\ref{thm:sltwoA2}, we obtain the following statement. 
\color{black}
\begin{corollary}
    \label{corol:seq_shares}
    Let $H_k\in (\Shares_k\setminus\Shares_{k-1})$ be a sequence of shares without arches, $k=0,1,2,3,\dots$. 
    Then any element $I\in \Shares$ is uniquely determined by the sequence of values $w_{\sltwo}((I, H_k))$.\end{corollary}
\begin{proof}
    Suppose we have two shares, $I$ and $I'$, for which the sequences under study coincide.
    By Lemma~\ref{corol:3:x^m+...}, $H_k = x^k + O(x^{k-1})$, thus shares $H_k$ form a basis in the free $\mathbb{C}[c]$-module $\Shares$, therefore $w_{\sltwo}((I, H)) = w_{\sltwo}((I', H))$ for every share $H\in \Shares$.
    So the shares $I$ and $I'$ are equivalent, and by definition they are equal as elements of $\Shares$.
\end{proof}

\subsection{Two-colored intersection graph}

The intersection graph of a share admits a natural two-coloring, which proves to be a useful tool in the study of the $\sltwo$-weight system on shares.
The \textbf{(two-colored) intersection graph}
of a share has white and black vertices corresponding to the arches and the bridges, respectively, and two vertices of any color are adjacent if and only if the corresponding chords intersect one another.



\begin{figure}[h!]
    \centering
    \includegraphics[width=50pt]{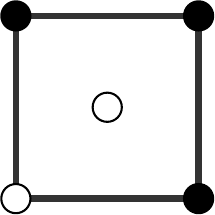}
    \caption{The two colored intersection graph of the share shown in Figure~\ref{fig:share_arc_brdg}}
    \label{fig:enter-label}
\end{figure}

Like in the case of chord diagrams, the value of the $\sltwo$-weight system depends on the two-colored intersection graph of a share only.
\begin{Theorem}
    The weight system $\sltwo$ on the vector space of shares over $\mathbb{C}$ depends on the two-colored intersection graph of a share rather than on the share itself.
\end{Theorem}
\begin{proof}
    Assume we have two shares $I_1$ and $I_2$ with isomorphic two-colored intersection graphs.
    Then the chord diagrams $(I_1, y^k)$ and $(I_2, y^k)$ have isomorphic intersection graphs, for each $k=0,1,2,\dots$, therefore, $w_{\sltwo}((I_1, y^k)) = w_{\sltwo}((I_2, y^k))$ by Claim~\ref{claim:ChL}.
    The value of the weight system $\sltwo$ on a given share $I$ is determined by the sequence $w_{\sltwo}((I, y^k))$ (see Corollary~\ref{corol:seq_shares}). 
    The shares $I_1$ and $I_2$ produce coinciding  sequences of values of~$w_\sltwo$, hence they are equal as elements of $\Shares$.
\end{proof}

The problem about existence of a natural extension of
the $\sltwo$ weight system to arbitrary graphs inspires
the following
\begin{question*}
    Does there exist a two-colored graph invariant satisfying two-colored four-term relations that coincides with the $\sltwo$ weight system on two-colored intersection graphs?
\end{question*}
Here by the two-colored four-term relations we mean relations similar to \eqref{eq:4term_graphs}, vertices of the graphs in which should obey the same changing
color rule as they do in the four-term elements in $A_2$.
An affirmative answer to this question would imply a
positive answer to Lando's question.

Under the assumption of existence of an extension, we can compute the value of the weight system $\sltwo$ on a cycle graph on five vertices. 
According to Claim~\ref{claim:bouchet}, the join of $C_5$ with a singleton graph is not an intersection graph (this join is shown in Fig.~\ref{fig:bouchet2}).

\begin{Lemma}
    Assuming there is an invariant of two-colored graphs which is an extension of the weight system $w_{\sltwo}$, its value on the $5$-cycle graph $C_5$ with all vertices colored black is given by the polynomial below:
    \label{lemma:C5}
    \begin{equation*}
        w_{\sltwo}(C_5) = y^5 - 10y^4 + 29y^3 + (5c^2-6c-26)y^2 + (-14c^2+8c+6)y + (c^3+5c^2).
    \end{equation*}
\end{Lemma}
For the proof, see Section~\ref{section:Appendix}.

\section{The involution of $\Shares$} 
\label{sec:involution}
In this section, we discuss the symmetry of $\Shares$, which reverses the order of the chord endpoints lying on one of the strands of a share. 
Such a reversion of one strand appears in Claim~\ref{claim:bouchet} as an example of the involution between two locally equivalent intersection graphs.
One can easily visualize this involution on the chord diagrams (see Figure~\ref{fig:Bouchet_chord}).
Pick any chord of a diagram; its endpoints split the boundary circle into two arcs.
By flipping one of the two arcs we replace the subgraph of the intersection graph induced by the neighborhood of the chord by its complement.

\begin{figure}[h!]
    \label{fig:Bouchet_chord}
    \begin{equation*}
        \raisebox{-33pt}{\includegraphics[width=70pt]{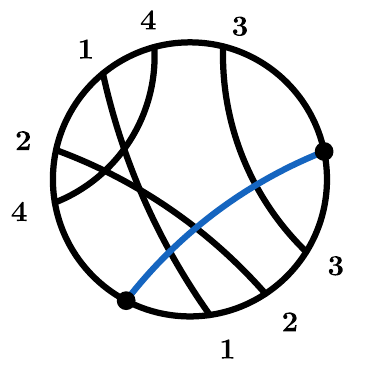}} \mapsto \raisebox{-33pt}{\includegraphics[width=70pt]{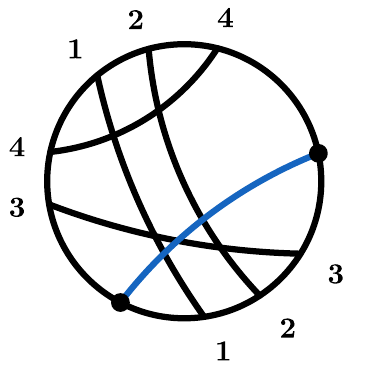}}
    \end{equation*}
    \caption{Two chord diagrams having locally equivalent intersection graphs. 
    The diagram on the right is obtained from 
    the first one by reversing the order of the chords' ends lying on one of the arcs}
\end{figure}

\begin{Def}
Let $I$ be a share with $m$ chords. 
Denote by $\overline{I}$ the share $I$ with reversed order of the ends of the chords lying on \emph{one} 
of the two strands (see Fig.~\ref{fig:share_flip}). 
Define the involution $\sigma \colon \Shares \to \Shares$ as follows:
\begin{equation*}
    \sigma(I) := (-1)^m \overline{I}, 
\end{equation*}
extended to~$\Shares$ by linearity.

\begin{figure}[h!]
    \begin{equation*}
        \sigma\colon \quad
        \raisebox{-30pt}{\includegraphics[width=40pt]{pic/share.pdf}} 
        \quad
        \mapsto \quad (-1)\cdot
        \raisebox{-30pt}{\includegraphics[width=40pt]{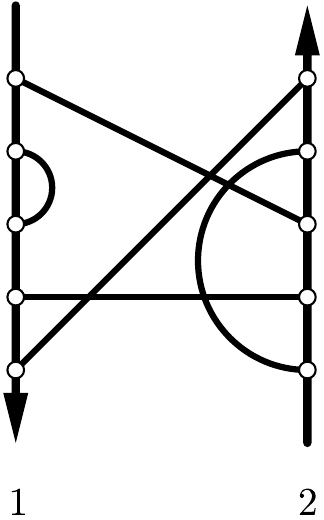}} \quad= \quad
        (-1)\cdot
        \raisebox{-30pt}{\includegraphics[width=40pt]{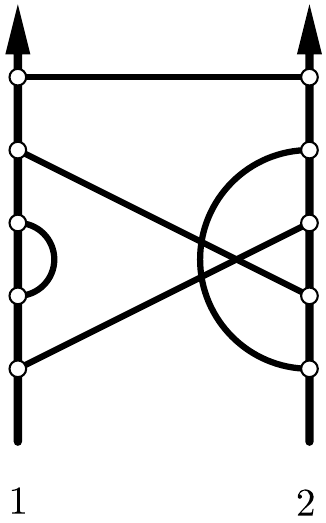}}
    \end{equation*}
    \caption{Reversion of the first strand of a share}
    \label{fig:share_flip}
\end{figure}
\end{Def}
The sign appearing in the definition is not important when we consider a share as an element of $A_2$, but it is essential in $\Shares$.
Recall that the complement graph to a graph $\Gamma$  is 
the graph $\overline{\Gamma}$ with the same set of vertices whose set of edges is complementary to that of $\Gamma$. The following statement is obvious.

\begin{Lemma}[] 
    \label{lemma:dual_share_dual_graph}
    Let $\Gamma_I$ be the intersection graph of a share $I$.
    Then the intersection graph $\Gamma_{\overline I}$ can be obtained from $\Gamma_I$ by replacing
    the subgraph  induced  by the black vertices
    with the complement subgraph.
    In particular, if $I$ has no arches, then $\Gamma_{\overline I}$ is
    the complement to $\Gamma_I$, $\Gamma_{\overline I}={\overline\Gamma}_I$.

    If we forget colors in the graph $\Gamma_{\overline{I}}$, then we obtain the intersection graph of the closure of $I$. 
\end{Lemma}

Now we need to check that the involution $\sigma$ is well-defined.

\begin{Lemma}
The involution $\sigma$ is well-defined, that is, 
\begin{enumerate}
\item the result of this operation does not depend on the choice of the strand;
\item if $D$ is a two-term, four-term, or six-term element, then 
$w_{\sltwo}(\overline{D}) =w_{\sltwo}({D}) = 0$.
\end{enumerate}
\end{Lemma}
\begin{proof}
\begin{enumerate}
    \item Let $I'$ and $I''$ be the two shares obtained from a given share $I$ by reversing the first and the second strand, respectively.
    Then $I'$ and $I''$ differ by a mutation reversing both arcs.
    Mutations do not change the value of the $\sltwo$ weight system~\cite{ChL}, hence $I'$ and $I''$ coincide as elements of $\Shares$.
    \item Both linear combinations of shares describing the leaf removal relation and four-term element remain the same up to a sign after reversing the orientation of one strand, while the first six-term element (which describes the Chmutov-Varchenko relations) transforms into the second one and vice versa.
\end{enumerate}
\end{proof}

The assertion below follows immediately from the definition of $\sigma$.
\begin{Lemma} 
\label{lem:eq5x^my^m}
The involution $\sigma$ relates the
two bases $x^m$ and $y^m$ in the following way:
    \begin{equation*}
        \sigma(x^m) = (-1)^m y^m, \qquad
        \sigma(y^m) = (-1)^m x^m.
    \end{equation*}
\end{Lemma}

\section{Operators of adding a chord}
In this section, we introduce three operators of adding 
a chord on $\Shares$ and describe their properties.

\subsection{Chord adding operator $U$ and its eigenbasis}
\label{sec:U_and_eigen}
Let us introduce \textbf{chord adding linear operators} $U$, $X$ and $Y$ acting on the space $\Shares$.
First we define their action on a single share in the following way:
    \begin{equation*}
        U\colon \raisebox{-30pt}{\includegraphics[width=40pt]{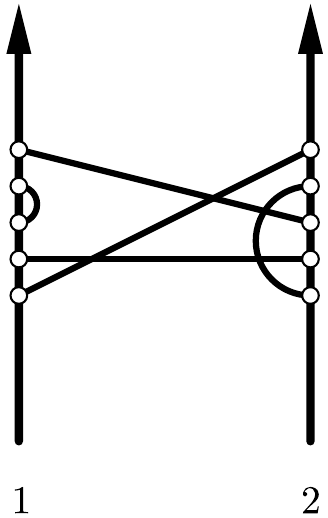}} \mapsto \raisebox{-30pt}{\includegraphics[width=40pt]{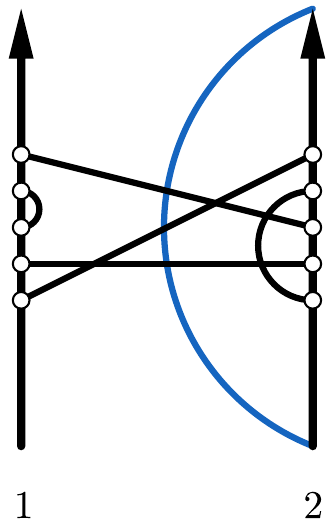}}, \quad
        X\colon \raisebox{-30pt}{\includegraphics[width=40pt]{pic/shareUXY.pdf}} \mapsto \raisebox{-30pt}{\includegraphics[width=40pt]{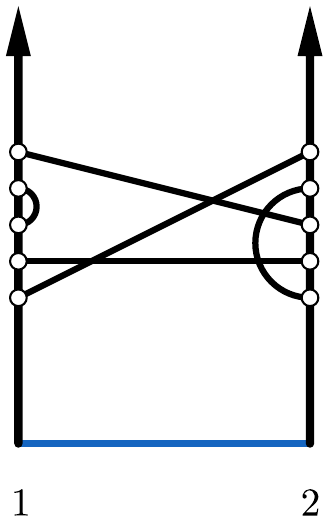}}, \quad
        Y\colon \raisebox{-30pt}{\includegraphics[width=40pt]{pic/shareUXY.pdf}} \mapsto \raisebox{-30pt}{\includegraphics[width=40pt]{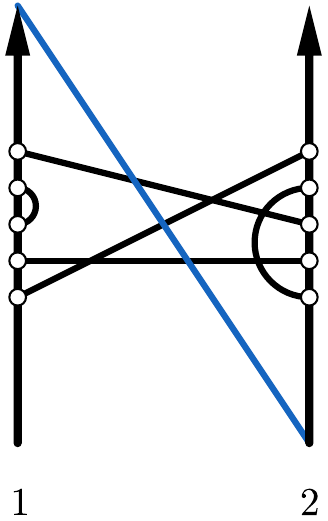}}.
    \end{equation*}
Each operator adds a chord whose endpoints coincide with the strand ends of a given share, and thus the action factorises correctly to the action on $\Shares$.
Indeed, if some linear combination of shares vanishes after adding all possible complement shares, then the same linear combination of shares with added chord will vanish after adding an arbitrary complement share: we can think of the added chord as being a member of the complement share.
Note also that the ``side'' where we add a chord does not matter, since the two-colored intersection graph of the resulting share does not depend on the choice of
the side.

\begin{Lemma}
    \label{lemma:commute_sigma_U}
    The involution $\sigma$ commutes with $U$,
    $\sigma U = U\sigma$, and acts on $X$ and $Y$ as follows:
    \begin{equation*}
        \sigma X = -Y \sigma, \qquad \sigma Y = -X \sigma.
    \end{equation*}
\end{Lemma}

The operators of adding a chord are subject to the following relation:
\begin{equation}
    \label{eq:4_term_operators}
    X - Y = U - c.
\end{equation}
This relation is a generalization of the four-term relation for shares. 
Suppose a share contains only one chord, and this chord is a bridge; then we obtain a conventional four-term relation.
Moreover, the fixed endpoint of an added arch should not necessarily lie next to the arch of a given share (see \cite{Chmutov_Duzhin_Mostovoy}). 

Equation \eqref{eq:4_term_operators} is not the only relation between the chord adding operators. 
We also have the equations
\begin{equation*}
    UX - XU = XY - YX = UY - YU,
\end{equation*}
which can be obtained from the four-term relation, and
\begin{equation}
    UY^2 = (2Y - 1)UY + (2c - Y - Y^2)U - (Y-c)^2, \label{eq:UY^2}
\end{equation}
which was derived from the Chmutov--Varchenko six-term relation in \cite{Zakorko}.
The latter equation yields the generating function for the sequence $U(y^m)$, which was obtained in \cite{KaZi}:
\begin{equation}
    \label{eq:u_gen_func}
    \sum_{m=0}^{\infty} U(y^m)t^m
    = \sum_{m=0}^{\infty} \sum_{i=0}^m u_{i,m}y^i t^m
    = \frac{1}{1-y t}\left( c+\frac{c^2 t^2-y t}{1-(2y-1)t-(2c-y^2-y)t^2}\right).
\end{equation}
By Lemma~\ref{lem:bridges_filtration}, $U y^m\in \Shares_m$, 
so for all $i>j$ we have $u_{i,j} = 0$. Therefore, $u_{i,i}$ are the eigenvalues of $U$. 
Moreover, it follows from \eqref{eq:u_gen_func} that
\begin{equation}
    \label{eq:umm}
    u_m := u_{m,m} = c- \frac{m(m+1)}{2}.
\end{equation}

Denote by $\Shares^{\lbrack m\rbrack}$ the eigenspace of $U$ with the eigenvalue $u_m$.
All the $u_i$ are pairwise distinct, hence there exists 
a basis such that $U$ is diagonal in it. Therefore, all the eigenspaces $S^{\lbrack m\rbrack}$, for $m \geq 0$,
are one-dimensional and the direct sum decomposition
\begin{equation}
    \label{eq:grading}
    \Shares = \bigoplus_{m=0}^{\infty} \Shares^{\lbrack m \rbrack}
\end{equation}
is a grading.
Note that grading \eqref{eq:grading} is consistent with filtration \eqref{eq:filtration}. 
For every $m$, there is the unique monic polynomial
in $y$ in $\Shares^{[m]}$; denote it by $e_m(y)$. The polynomials $e_m(y), m = 0, 1, 2, \ldots$, form a basis in $\Shares$, which is an eigenbasis for $U$.

Since $e_n = y^n + O(y^{n-1})$, we can reformulate Lemma~\ref{corol:3:x^m+...} as follows.

\begin{corollary} 
\label{corol:11:e_m+...}
Let $I$ be a share with m bridges and without arches. 
Then $I = e_m + \sum_{i=0}^{n-1} \alpha_i(c)e_i$, where $\alpha_i(c)$ are some polynomials. 
\end{corollary}

For the future investigation of the basis $e_n$, we need to discuss additional structures on $\Shares$.

\subsection{A bilinear form on $\Shares$} \label{sec:bilinear}
The main result of this section is an explicit formula for the basic elements $e_m(c)$. 
\begin{Def}
Define a bilinear form $\langle \cdot, \cdot \rangle \colon \Shares \to \mathbb C[c]$,
where we consider $\Shares$ as a $\mathbb{C}[c]$-module,
as follows:
\begin{equation*}
\langle I,H \rangle := w_{\sltwo} (( I, H)), \text{ where } I,H \in A_2.
\end{equation*}
\end{Def}
Note that this bilinear form maps a pair of shares $I$ and $H$ containing no arches to the value of $w_{\sltwo}$ on the join of the intersection graphs of the shares.

\begin{Lemma} 
Operator $Y$ is adjoint to $U$, and $X$ is a self-adjoint operator,
with respect to the bilinear form $\langle\cdot,\cdot\rangle$:
\begin{equation*} 
\langle Y v_1, v_2 \rangle = \langle v_1, U v_2 \rangle,  \quad \langle X v_1, v_2 \rangle = \langle v_1, X v_2 \rangle, \quad \text{where $v_1, v_2 \in \Shares$. }
\end{equation*}
\end{Lemma}
\begin{proof}
It suffices to check the equations only for the case where both $v_1$ and $v_2$ are shares.
\end{proof}
The formula below is useful in the case where we know the decomposition of a share in the basis $\{y^n\}$.
\begin{Lemma}
    \label{lemma:Ie}
    Given an element $I\in\Shares$ we have the following expression for the inner product of $I$ and $e_n$:
    \begin{equation*}
        \langle I(y), e_n(y) \rangle = I(u_n) e_n(c),
    \end{equation*}
where $u_n$ are given by Eq.~\eqref{eq:umm}.    
\end{Lemma}

\section{Proof of the main theorem} \label{sec:proof_main}
In this section, we prove Theorem~\ref{thm:main_graphs}.

\begin{Theorem}[\cite{KaZi}] 
\label{thm:5}
For every share $I$
containing no arches the generating function 
$G_{I}(t) := \sum_{i=0}^{\infty} \langle I, y^n \rangle t^n$ has the form
\begin{equation*}
    G_{I}(t) = \sum_{k=1}^{m} \frac{r_k^{(I)}(c)}{1-u_k t},  
\end{equation*}
where $u_k = u_{k,k} 
= c - \frac{k(k+1)}{2}
$ are the diagonal coefficients of $U$ given by Eq.~(\ref{eq:umm}), and $m$ is the number of bridges in $I$.
\end{Theorem}

Since $\langle y^n, \mathds{1} \rangle = c^n$ and for every $I\in\Shares$ we have $\langle I, y^n \rangle = \langle U^n I, \mathds 1 \rangle$, we obtain
\begin{Lemma} We have
    \label{lemma:eqs9}
    \begin{equation*}
        \langle e_m, y^n \rangle = e_m(c) u_m^n,
    \end{equation*}
    and, therefore,
    \begin{equation*}
        G_{e_m}(t) = \frac{e_m(c)}{1-u_m t}.
    \end{equation*}
\end{Lemma}

\begin{corollary}
    \label{corollary:generating_ec}
    Let $I$ be an arbitrary element of $A_2$, not necessarily without arches. 
    Then we have 
    \begin{equation*}
        G_I(t) = \sum_{k=0}^m \frac{a_k^{(I)}(c) e_k (c)}{1-u_k t}, 
    \end{equation*}
    where $I = \sum_{k=0}^m a_k^{(I)}(c) e_k(y)$
    is the  unique decomposition. 
\end{corollary}

An explicit formula for $e_i(c)$ is given in Lemma~\ref{lem:e_m(c)}.
Therefore, in order to find the generating function $G_I(t)$ for some share $I$, it is sufficient to find the decomposition of $I$ with respect to the basis $e_0,e_1,\ldots,e_m$.

\begin{proof}[Proof of Theorem~\ref{thm:main_graphs}]
    The operator $U$ commutes with $\sigma$, therefore, $U (\sigma e_m) = u_m \cdot\sigma e_m$. 
    Thus $\sigma e_m$ also is an eigenvector of $U$,
    with the same eigenvalue. 
    It is collinear to $e_m$, since $\Shares^{[m]}$ is one-dimensional.
    
    The share $y^m$ belongs to $\Shares_m$, whence $y^m = x^m + O(x^{m-1})$.
    Now we can compute the leading term of $e_m(x)$ in~$y$:
    \begin{equation*}
        \sigma e_m = \sigma (x^m + O(x^{m-1})) = (-1)^m y^m + O(y^{m-1}),
    \end{equation*} 
    and hence $\sigma e_m = (-1)^m e_m$.
       
    Let $\Gamma$ be the intersection graph of the chord diagram obtained by closing a share $I$ without arches. If $I = \sum_{m=0}^{k} \alpha^{(I)}_m(c) e_m(y)$, then $\sigma I = \sum_{m=0}^{k} (-1)^m \alpha^{(I)}_m(c) e_m(y)$, hence $\alpha^{(I)}_m = (-1)^m \alpha^{(\sigma I)}_m$. 
    
    Therefore, theorem now follows from Lemma~\ref{lemma:dual_share_dual_graph} and the fact that $r_m^{\Gamma} = \alpha_m e_m(c)$.
\end{proof}

\section{Appendix}
\label{section:Appendix}

In this section we start with completing the proof of Theorem~\ref{thm:sltwoA2} and then prove certain additional useful properties of the bases and the chord adding operators
in the algebra of shares and discuss several applications.

\subsection{Algebraic independence of $x$, $c_1$ and $c_2$}

(End of the proof of Theorem~\ref{thm:sltwoA2})
For the second part of the proof, that is, for the proof of algebraic independence of
the elements $c_1,c_2,x$ in ${\rm Im}~w_{\sltwo}$, it is convenient for us to use another basis in $\sltwo$ instead of $x_1, x_2, x_3$, namely
\begin{equation*}
    e = \begin{pmatrix}
        0 & 1 \\
        0 & 0
    \end{pmatrix}, \quad
    f = \begin{pmatrix}
        0 & 0 \\
        1 & 0
    \end{pmatrix}, \quad
    h = \begin{pmatrix}
        1 & 0 \\
        0 & -1
    \end{pmatrix}.
\end{equation*}
In this basis, the Casimir element $c = x_1^2 + x_2^2 + x_3^2$ is given by the expression $\frac{1}{2}\left(ef + fh + \frac{1}{2} h^2\right)$.
Therefore, we can instead show the algebraic independence of the following elements:
\begin{equation*}
    c'_1 = \left(ef + fh + \frac{1}{2} h^2\right)\otimes 1, \quad
    c'_2 = 1\otimes \left(ef + fh + \frac{1}{2} h^2\right), \quad
    x' = e\otimes f + f\otimes e + \frac{1}{2} h\otimes h.
\end{equation*}
Denote by $\Lambda=\Lambda(a_1, a_2)=\bigoplus_{i=0}^{\infty} \Lambda^n$ the graded ring of polynomials in two variables $a_1$ and $a_2$, where $\Lambda^n$ is the subspace of homogeneous polynomials of degree $n$.
Consider the representation $\rho\colon \sltwo\to \Hom(\Lambda)$, which acts on $\Lambda$ by the following vector fields:
\begin{equation*}
    \rho(e) = a_1 \partial_{a_2}, 
    \quad 
    \rho(f) = a_2 \partial_{a_1},
    \quad
    \rho(h) = a_1 \partial_{a_1} - a_2 \partial_{a_2}.
\end{equation*}
Note that every $\Lambda^n$ is invariant under the action of $\rho$.
Restrict our attention to these invariant subspaces.

Using the equality $U(\sltwo)\otimes U(\sltwo) = U(\sltwo\oplus\sltwo)$, construct the representation $\rho'$ of $U(\sltwo\oplus\sltwo)$ as the tensor product of two copies of $\rho$ extended to an algebra representation of $U(\sltwo)$,
\begin{equation*}
    \rho'\colon U(\sltwo\oplus \sltwo) \to \Hom(\Lambda(a_1, a_2)\otimes\Lambda(b_1, b_2)).
\end{equation*} 
The elements $c'_1$ and $c'_2$ act on polynomials in the following way:
\begin{equation*}
    \rho'(c'_1) = 
    \frac{1}{2}\left(a_1 \partial_{a_1} + a_2 \partial_{a_2}\right) \left(a_1 \partial_{a_1} + a_2 \partial_{a_2} + 2\right),
\end{equation*}
\begin{equation*}
    \rho'(c'_2) = 
    \frac{1}{2}\left(b_1 \partial_{b_1} + b_2 \partial_{b_2}\right) \left(b_1 \partial_{b_1} + b_2 \partial_{b_2} + 2\right).
\end{equation*}
Therefore, the subspace $\Lambda^{n_1}(a_1,a_2)\otimes\Lambda^{n_2}(b_1,b_2)$
spanned by the homogeneous polynomials of bidegree $(n_1, n_2)$ in~$a,b$  is an eigenspace for each $c_1'$ and $c_2'$ with eigenvalues $\alpha=n_1(n_1+2)/2$ and $\beta=n_2(n_2+2)/2$ respectively.
One can check that the homogeneous polynomial $a_1^{n_1 - k} a_2^{k} (a_1b_2-a_2b_1)^k$ of bidegree $(n_1, n_2)$, where $k \leqslant \min(n_1, n_2)$, is an eigenvector of $\rho'(x')$ with the eigenvalue $(\frac{1}{2}(n_1+n_2+1) - k)^2 - \frac{1}{4}(n_1(n_1 + 2) + n_2(n_2 + 2) + 1)$.

Now let us finally proceed to algebraic independence.
Assume $x'$, $c_1'$ and $c_2'$ are algebraically dependent, so
that there is a polynomial $P$ such that $P(x', c_1', c_2') = 0$.
Applying $\rho'$ to $P$, we see that for any eigenvalue $\lambda$ of $\rho'(x')$ that corresponds to an eigenvector in $\Lambda^{n_1}\otimes\Lambda^{n_2}$ we have $p_{\alpha\beta}(\lambda):=P(\lambda, \alpha, \beta)=0$. 
Then for some large $n_1$ and $n_2$ the new one-variable polynomial $p$ has too many roots.
Indeed, if the degree of $P$ in the first argument is not greater than $m$, then for $n_1$, $n_2$ such that $\min(n_1, n_2) \geqslant m$ the polynomial $p_{\alpha\beta}$ has more than $m$ roots, thus should be identical zero.

\subsection{An orthogonal basis in $\Shares$}
For a nondegenerate symmetric bilinear form, one can construct an orthogonal basis with respect to this form.
In the case of our bilinear form $\langle \cdot, \cdot \rangle$, we have a simple formula for such a basis.
\begin{Theorem}
    The sequence of polynomials $p_n$ given by
    \begin{equation*}
        p_n := \prod_{m=0}^{n-1} \left(y - u_m\right) =  \prod_{m=0}^{n-1} \left(y - c + \frac{m(m+1)}{2}\right)
    \end{equation*}
    forms an orthogonal basis in $\Shares$ w.r.t. $\langle \cdot, \cdot \rangle$.
\end{Theorem}
\begin{proof}
    The most convenient way to prove the theorem is to show that each $p_n$ is orthogonal to all $e_k$, for $k<n$. 
    The inner product of $p_n$ and $e_m$ can easily be computed  through Lemma~\ref{lemma:Ie}:
    \begin{equation*}
        \langle p_n, e_k \rangle = p_n(u_k) e_k(c) = e_k(c) \prod_{m=0}^{n-1} \left(u_k - u_m\right).
    \end{equation*} 
    If $k\ge n$, then none of the factors in the product is zero,
    while for $k<n$ the factor corresponding to $m=k$ vanishes.
 \end{proof}

\begin{Lemma}
    The chord adding operators have the following form in the orthogonal basis $p_n$:
    \begin{align}
        \label{eq:Xyn_rec}
        Xp_{n} &= p_{n+1} + \left(c - n(n+1)\right) p_{n} - n^{2} \left(c - \frac{n^{2}-1}{4}\right) p_{n-1},\\
        Yp_{n} &= p_{n+1} + \left(c - \frac{n(n+1)}{2}\right) p_n,\\
        \label{eq:Yyn_rec}
        Up_{n} &= \left(c - \frac{n(n+1)}{2}\right) p_{n} - n^{2} \left(c - \frac{n^{2}-1}{4}\right) p_{n-1}.
    \end{align}
\end{Lemma}
\begin{proof}
    The expression for $X$ was proven in \cite{Zakorko}.
    The formula for operator $Y$ is just the definition of $p_n$. 
    The last expression can be obtained through the four-term relations \eqref{eq:4_term_operators} for the operators.
\end{proof}

\begin{corollary} We have
    \begin{equation*}
        \langle p_n, p_n \rangle = (-1)^n (n!)^2 \prod_{m=1}^{n} \left(c - \frac{m^2-1}{4}\right).
    \end{equation*}
\end{corollary}
\begin{proof}
    \begin{equation*}
        \langle p_n, p_n \rangle = \langle p_n, Xp_{n-1} \rangle = \langle Xp_n, p_{n-1} \rangle = -n^2\left(c - \frac{n^2 - 1}{4}\right) \langle p_{n-1}, p_{n-1} \rangle.
    \end{equation*}
\end{proof}

\subsection{Some properties of $e_n$}
\label{sec:props_e_n}
The basis $e_n$ was introduced quite abstractly as an eigenbasis for the operator $U$. 
Although we can compute $e_n(y)$ explicitly, we do not know any explicit expressions of the basic elements 
in terms of linear combinations of shares having some combinatorial meaning, like averaging over all shares of a particular type. 
\begin{Lemma}
\label{lem:e_m(c)} We have
\begin{equation*}
e_n(c) = \frac{n!}{(2n-1)!!}\prod_{m=1}^n\left( c - \frac{m^2 - 1}{4} \right).
\end{equation*}
\end{Lemma}
\begin{proof}
    There is one more way to compute the inner product $\langle p_n, p_n \rangle$ using that $p_n(y) = e_n(y) + O(y^{n-1})$:
    \begin{equation*}
        \langle p_n, p_n \rangle = \langle p_n, e_n \rangle = p_n(u_n) e_n(c).
    \end{equation*}
    It remains to compute $p_n(u_n)$:
    \begin{multline*}
        p_n(u_n) = 
        \prod_{m=0}^{n-1} (u_n - u_m) = 
        \prod_{k=1}^{n} \frac{((n-k)(n-k+1)-n(n+1))}{2} = 
        \prod_{k=1}^{n} \frac{(-k)(2n+1-k)}{2} = \\
        (-1)^n n! \prod_{k=1}^{n} \frac{(2n+1-k)}{2} = 
        (-1)^n n! \frac{(2n)!}{2^n n!} = 
        (-1)^n n! (2n-1)!!,
    \end{multline*}
    and we are done.
\end{proof}

\begin{Theorem}
    \label{thm:operators_e}
    The chord adding operators have the following form in the basis $e_n$:
    \begin{align}
        \label{eq:xen}
        X e_n &= e_{n + 1} - \frac{n (n + 1)}{4} e_{n} + \frac{n^2}{4n^2-1} \left(c - \frac{n^2 -1}{4}\right)^2 e_{n - 1},\\
        Y e_n &= e_{n + 1} + \frac{n (n + 1)}{4} e_{n} + \frac{n^2}{4n^2-1} \left(c - \frac{n^2 -1}{4}\right)^2 e_{n - 1},\\
        U e_n &= \left(c - \frac{n(n+1)}{2} \right) e_{n}.
    \end{align}
\end{Theorem}

\begin{proof}
It is sufficient to prove only the first formula, the second one is dual to it, and the third follows directly from the definition of $e_n$.

By Corollary~\ref{corol:seq_shares}, we need to check that the pairing of $y^n$ with the right-hand side of \eqref{eq:xen} coincides with that with the left-hand side.
Now, let us compute the generating function
\begin{equation*}
    F_n(t) := \sum_{k=0}^{\infty} \langle y^k, Y e_n\rangle t^k.
\end{equation*}
The operators $Y$ and $U$ are adjoint, therefore we can rewrite $F_n$ in the following way:
\begin{equation*}
    F_n(t) = \sum_{k=0}^{\infty} \langle Uy^k, e_n\rangle t^k.
\end{equation*}
Now we can write out the exact expression for this generating function, substituting $y=u_n$ in the generating function \eqref{eq:u_gen_func} for $Uy^n$.
\begin{equation*}
    F_n(t) = \frac{c + (c-(2c+1)u_n) t + c(u_n^2 + u_n -c^2) t^2 }{(1-u_n t)(1 - (2u_n - 1)t - (2c-u_n-u_n^2) t^2)} \cdot e_n(c).
\end{equation*}
We leave it for the reader to verify that the generating function of the right-hand side indeed equals $F_n(t)$.
This computation involves the exact formula for $e_n(c)$ given in Theorem~\ref{lem:e_m(c)}, and the value of the bilinear form $\langle y^k, e_m\rangle$ obtained in Lemma~\ref{lemma:Ie}.
\end{proof}

\begin{corollary}
    For the decomposition of $e_{n+1}$ in terms of the basis $\{y^k\}$, we have the following recurrent formula:
    \begin{equation*}
        e_{n + 1}(y) = \left(y-\frac{n (n + 1)}{4}\right) e_{n}(y) - \frac{n^2}{4n^2-1} \left(c - \frac{n^2 -1}{4}\right)^2 e_{n - 1}(y).
    \end{equation*}
\end{corollary}
\begin{proof}
    This follows directly from \eqref{eq:xen}.
\end{proof}

\subsection{Proof of Lemma~\ref{lemma:C5}}
The formula below was demonstrated in \cite{Fil22}:
\begin{equation*}
    \begin{aligned}
        w_{\sltwo}\left(\left(C_5, n\right)\right)= & \frac{1}{630} c\left(\left(270 c^4-540 c^3-999 c^2+576 c+324\right)(c-1)^n\right. \\
        + & \left(280 c^4-1610 c^3+3234 c^2-2646 c+756\right)(c-6)^n \\
        & \left.\quad+\left(80 c^4-1000 c^3+4065 c^2-6120 c+2700\right)(c-15)^n\right).
    \end{aligned}
\end{equation*}
Under the assumption that there is an extension of the weight system $\sltwo$ to the space of two-colored graphs, the generation function of $w_{\sltwo}\left(\left(C_5, n\right)\right)$ should be in the form given by Corollary~\ref{corollary:generating_ec}. 
Therefore, under the notation of this corollary, we can extract the coefficients $a_k^{(C_5)}(c)$, which leads to the desired value of $w_{\sltwo}(C_5)$.

\subsection{Relation to complete graphs}
\label{sec:rel_to_complete}
In this section we use $\sigma$ and our main result (Theorem~\ref{thm:main_graphs}) to obtain a relation between the values of the $\sltwo$ weight system on complete graphs and on complete bipartite graphs.
\begin{Theorem} \label{thm:K_m}
Denote by $w_{\sltwo}(K_m)$ the value of the $\sltwo$ weight system on the complete graph with $m$ vertices {\rm(}note that the latter equals $y_m(c)$\rm{)}.
Then
\begin{equation}
w_{\sltwo}(K_m) = c^m - 2 \sum_{i=0}^{m-1} \sum_{\substack{j\leq i,\\ |m-j| \mbox{ is odd}}} \frac{u_{i,m} \cdot r_i^{(j)}(c)}{u_j - u_m},
\end{equation}
where
\begin{itemize}
    \item $u_j - u_m = (c-\frac{j(j+1)}{2}) - (c-\frac{m(m+1)}{2}) = \frac{1}{2}(m(m+1) - j(j+1))$;
    \item $u_{i,m}(c)$ 
    are given by the generating function~\eqref{eq:u_gen_func};
    \item $r_i^{(j)}(c)$ is the coefficient of $\frac{1}{1- u_j t}$ in the decomposition~\eqref{eq:fromKaZi} for a particular case where $\Gamma$ is the empty graph.
\end{itemize}
\end{Theorem}
\begin{Lemma} 
    The coefficients $u_{i,j}$ of $U$ with respect to the basis $y^n$ and the coefficients $\overline u_{i,j}$ of $U$ with respect to the basis $x^n$, $i,j,n = 0,1,2,\ldots$ are related as follows:
    \begin{equation*}
        \overline u_{i,m} = (-1)^{i+m}u_{i,m}.
    \end{equation*}
\end{Lemma}

\begin{proof}
This follows from the equation $\sigma(y^m) = (-1)^m x^m$ and the fact that $U$ commutes with $\sigma$ (see Lemmas~\ref{lem:eq5x^my^m} and \ref{lemma:commute_sigma_U}, respectively):
\begin{multline*}
    \sum_{i=0}^m \overline{u}_{i,m}x^i = Ux^m = U(-1)^m \sigma(y^m) = (-1)^m\sigma(U(y^m)) =  \\
    (-1)^m \sum_{i=0}^m u_{i,m} \sigma(y^i) = \sum_{i=0}^m (-1)^{m+i}u_{i,m}y^i.
\end{multline*}
\end{proof}

\begin{Lemma}\label{lemma:Split}
    We have
    \begin{equation}
        \Split_m(t) = \frac{w_{\sltwo}(K_m)+t\sum_{i=0}^{m-1} (-1)^{m-i} u_{i,m}\Split_i(t)}{1-u_m t}, \label{eq:SplitrecTh}
        \end{equation}
    where $\Split_m$ is the generating function of the values of $w_{\sltwo}$ on split graphs: $\Split_m = \sum_{n=0}^{\infty} (x^m, y^n) t^n$.
\end{Lemma}
\begin{proof}
    The proof is the same as the proof of the analogous theorem (Theorem 3) in~\cite{KaZi}.
\end{proof}
\begin{proof}[Proof of Theorem~\ref{thm:K_m}]
It was proven in~\cite{KaZi} that
\begin{equation}
    \CB_m(t) = \frac{c^m+t\sum_{i=0}^{m-1}u_{i,m}\CB_i(t)}{1-u_m t}, \label{eq:CBrecTh}\\
\end{equation}
It follows from Theorem~\ref{thm:main_graphs} that if we set
\begin{equation}
    \CB_i = \sum_{j=0}^i r_i^{(j)} \frac{1}{1-u_j t} \label{eq:CBind}
\end{equation}
then
\begin{equation}
    \Split_i = \sum_{j=0}^i (-1)^{i-j} r_i^{(j)} \frac{1}{1- u_j t}, \label{eq:Splitind}
\end{equation}
Substitute \eqref{eq:CBind}, \eqref{eq:Splitind} to \eqref{eq:CBrecTh}, \eqref{eq:SplitrecTh}:
\begin{align*}
    \CB_m(t) &= \frac{c^m}{1-u_m t} + t\left(\sum_{i=0}^{m-1} u_{i,m} \sum_{j=0}^i r_i^{(j)} \frac{1}{(1-u_j t)(1-u_m t)} \right),\\
    \Split_m(t) &= \frac{w_{\sltwo}(K_m)}{1-u_m t} + t\left(\sum_{i=0}^{m-1} (-1)^{m-i} u_{i,m} \sum_{j=0}^i (-1)^{i-j} r_i^{(j)} \frac{1}{(1-u_j t)(1-u_m t)} \right).
\end{align*}

Since
\begin{equation*}
    \frac{t}{(1-u_j t)(1- u_m t)} = \frac{1}{u_j - u_m}\left(\frac{1}{(1- u_j t)} - \frac{1}{(1- u_m t)} \right),
\end{equation*}
we rewrite 

\begin{align*}
    \CB_m(t) &= \frac{c^m}{1-u_m t} + \left(\sum_{i=0}^{m-1} \sum_{j=0}^i  u_{i,m} r_i^{(j)} \frac{1}{u_j - u_m} \left(\frac{1}{1- u_j t} - \frac{1}{1- u_m t} \right) \right), \label{eq:CB_frac_jm}\\
    \Split_m(t) &= \frac{w_{\sltwo}(K_m)}{1-u_m t} +\left(\sum_{i=0}^{m-1} \sum_{j=0}^i (-1)^{m-j}  u_{i,m} r_i^{(j)} \frac{1}{u_j - u_m} \left(\frac{1}{1- u_j t} - \frac{1}{1- u_m t} \right) \right).
\end{align*}

It follows from Theorem~\ref{thm:main_graphs} that the coefficients of $\frac{1}{1-u_m t}$ in $\CB_m(t)$ and in $\Split_m(t)$ are equal to one another. This proves the theorem.

\end{proof}

\end{document}